\documentclass[10pt,twocolumn,letterpaper]{article}
\usepackage{authblk}
\usepackage{multirow}
\usepackage{amsfonts}
\usepackage{amsmath,amssymb,mathrsfs,amsthm}
\usepackage{color}

\usepackage{hyperref}
\usepackage{epsfig}
\usepackage{epstopdf}
\usepackage{array,booktabs}
\usepackage{tabularx}
\usepackage{algorithm}
\usepackage[noend]{algcompatible}

\usepackage{comment}
\usepackage{bm}
\usepackage{graphicx}
\usepackage{subcaption}

\newtheorem{lemma}{Lemma}
\newtheorem{theorem}{Theorem}
\newtheorem{remark}{Remark}

\newtheorem{assumption}{Assumption}
\newtheorem{corollary}{Corollary}

\begin{document}

\title{Progressive Power Homotopy for Non-convex Optimization}

\author[1]{Chen Xu}
\affil[1]{Department of Engineering, Shenzhen MSU-BIT University, China. 
\authorcr Email: xuchen@smbu.edu.cn
}

\date{}

\maketitle

\begin{abstract}
We propose a novel first-order method for nonconvex optimization of the form $\max_{\bm{w}\in\mathbb{R}^d}\mathbb{E}_{\bm{x}\sim\mathcal{D}}[f_{\bm{w}}(\bm{x})]$, termed Progressive Power Homotopy (Prog-PowerHP). The method applies stochastic gradient ascent to a surrogate objective obtained by first performing a power transformation and then Gaussian smoothing, $F_{N,\sigma}(\bm{\mu}):=\mathbb{E}_{\bm{w}\sim\mathcal{N}(\bm{\mu},\sigma^2I_d),\bm{x}\sim\mathcal{D}}[e^{Nf_w(\bm{x})}]$, while progressively increasing the power parameter $N$ and decreasing the smoothing scale $\sigma$ along the optimization trajectory. We prove that, under mild regularity conditions, Prog-PowerHP converges to a small neighborhood of the global optimum with an iteration complexity scaling nearly as $O(d^2\varepsilon^{-2})$. Empirically, Prog-PowerHP demonstrates clear advantages in phase retrieval when the samples-to-dimension ratio approaches the information-theoretic limit, and in training two-layer neural networks in under-parameterized regimes. These results suggest that Prog-PowerHP is particularly effective for navigating cluttered nonconvex landscapes where standard first-order methods struggle.

%It power-transforms the objective to put more weight on higher local maxima and then perform Gaussian smoothing to remove poor local maxima. Stochastic gradient ascent is performed to maximize the transformed objective, during which the power is progressively increased.
\end{abstract}


\section{Introduction}
In this work we aim to solve the following stochastic optimization problem
\begin{equation}
\label{stochastic-objective}
\max_{\bm{w}\in\mathbb{R}^d} G(\bm{w}):= \mathbb{E}_{\bm{x}\sim \mathcal{D}}[f_{\bm{w}}(\bm{x})],
\end{equation}
where $\mathcal{D}$ is a known distribution, $f$ is a known fitness function with accessible derivatives, and the objective $G$ has a unique global maximum point $\bm{w}^*$ and possibly multiple local maximums. If $\bm{x}$ represents data and $f_{\bm{w}}$ equals negative of the individual loss function associated with a model with $\bm{w}$ being parameters, then (\ref{stochastic-objective}) becomes the loss-minimization problem for model training in machine learning. The most popular methods for the loss-minimization problem are sample-based gradient methods, such as the full-batch gradient descent (GD), stochastic gradient ascent (SGD, \cite{bottou2010large}), and Adaptive Moment Estimation (Adam, \cite{kingma2017adam}). While the full-batch GD is guaranteed to approach the stationary points of any sample estimate of $G(\bm{w})$ under an appropriate learning schedule, it may get stuck near shallow local minima and saddle points. The randomness introduced by mini-batch selection helps SGD mitigate this issue and boosts performances. While the performance of SGD significantly depends on the learning rate schedule, Adam is less susceptible and converges much faster. However, these optimizers are still often trapped near local optimums. Hence, the optimization problem of locating the global optimum still remains open, which is crucial for training under-parameterized models.

A well-accepted explanation on SGD's ability to escape from narrow local minimums is that the mini-batch selection randomness makes SGD actually optimize the smoothed objective $\mathbb{E}_{\bm{w}}[G(\bm{w})]$ of $G(\bm{w})$, which eliminates the narrow local minimums \cite{Kleinberg2018}. This relates to the smoothing technique (\cite{Mobahi2012}) used by another type of popular derivative-free non-convex optimization methods---homotopy for optimization (e.g., \cite{MobahiFisher2015, Hazan2016}). Its standard version solves the deterministic problem $\min_{\bm{w}}h(\bm{w})$ by minimizing the Gaussian-smoothed surrogate $H_{\sigma}(\bm{\mu}):=\mathbb{E}_{\bm{\varepsilon}\sim\mathcal{N}(\bm{0}, I_d)}[h(\mu+\sigma \bm{\varepsilon})]$. Specifically, it incrementally decreases $\sigma$ in an outer loop, and optimizes the surrogate in the inner loop under the current $\sigma$ value, taking the solution found in the previous inner loop as the initial candidate for the current inner loop. Homotopy works because the smoothing technique removes sharp local minimums and the decreasing-$\sigma$ mechanism drags the global optimum of the surrogate towards that of the original objective $h(\bm{w})$ under certain conditions (\cite{Hazan2016}). 

Traditional homotopy methods are often hindered by the inefficiency of their double-loop structure and a notorious sensitivity to the $\sigma$-decay schedule, which is notoriously difficult to tune. To address these challenges, recent research has pivoted toward more time-efficient single-loop mechanisms and strategies to minimize the displacement of the optimum caused by the smoothing process. For instance, the Single-Loop Gaussian Homotopy (SLGH) algorithm \cite{Iwakiri2022} introduces an automated $\sigma$-decay mechanism. GS-PowerOpt \cite{GS-PowerOpt} utilizes a power transformation prior to Gaussian smoothing to effectively reduce the distance between the pre- and post-smoothed optima, while GS-PowerHP \cite{GSPowerHP} integrates both the power transformation and a $\sigma$-decaying schedule. Despite these advancements, significant gaps remain for solving (\ref{stochastic-objective}): SLGH lacks guarantees beyond convergence to a standard local optimum, while GS-PowerOpt and GS-PowerHP are restricted to zeroth-order, deterministic settings.

In this paper, we introduce Progressive Power Homotopy (Prog-PowerHP), a novel smoothing-based framework designed to solve the stochastic optimization problem defined in (\ref{stochastic-objective}). It offers several critical advantages over existing optimization frameworks. Compared with the standard first-order methods of SGD and Adam and the homotopy method of SLGH, which are frequently susceptible to local optima, our method provides a rigorous theoretical pathway to approximating the global optimum under mild conditions (see Corollary \ref{wrap-up}). By utilizing a single-loop structure, it bypasses the excessive computational overhead and complex sub-problem convergence requirements of traditional double-loop homotopy. Compared with GS-PowerOpt and GS-PowerHP, our algorithm is designed for the stochastic optimization problem and adopts a gradually increasing power for better performance. A key innovation of our approach is the progressive power mechanism, which dynamically increases the power transformation as the smoothing parameter diminishes, effectively refining the search space and prioritizing deeper basins of attraction as the landscape's fine-grained features emerge.\\

\noindent\textbf{Contribution.} The primary contributions of this work are three-fold. To the best of our knowledge, this is the first work to combine homotopy continuation with a progressive power-transformation of the objective for stochastic first-order optimization. Second, we establish the theoretical foundations of this approach by providing a comprehensive convergence analysis for Prog-PowerHP in the stochastic setting. Third, through extensive empirical evaluation on two-layer ReLU networks and phase retrieval, we demonstrate the algorithm's superior landscape navigation capabilities. Our results show that in under-parameterized and limited-sample regimes—where standard optimizers consistently settle into higher-loss spurious minima—Prog-PowerHP achieves significantly higher approximation efficiency by locating deeper, more optimal local basins.\\

\noindent\textbf{Related Work.} There are existing studies that also exponentially transform the loss. For example, in tilted empirical risk minimization, this technique is used to serve multiple purposes, such as to decrease the influence of sample outliers and to improve model generalizations (\cite{LiBeirami2023,aminian2025}). In the field of safe reinforcement learning, this technique is used to emphasize the worst-case scenarios (\cite{Dvijotham2014,fei2020risk,fei2021exponential}). The main difference between these works and ours is, they treat the transformed objective as a risk-sensitive criterion, intentionally redefining optimality to emphasize tail performance, robustness, or safety. In contrast, our work uses a Gaussian-smoothed power-transformed objective purely as an optimization surrogate, designed to preserve the maximizer of the original fitness while improving concentration and optimization geometry.
 
Our method is also fundamentally distinct from existing homotopy-based approaches. Unlike gradient-based standard homotopy methods \cite{Hazan2016} and SLGH \cite{Iwakiri2022}, our approach employs a power-transformation mechanism. Moreover, in contrast to GS-PowerOpt \cite{GS-PowerOpt}, GS-PowerHP \cite{GS-PowerHP}, and the path-learning homotopy method \cite{Lin2023}—all of which are developed for deterministic optimization—our framework is tailored to stochastic first-order optimization, where gradient noise and limited sample access fundamentally alter the optimization dynamics.

\section{Prog-PowerHP: The Proposed Method}
\subsection{Motivations}
The design of our algorithm is motivated by a simple geometric principle: emphasize dominant extrema while suppressing spurious ones. Specifically, the power transformation amplifies higher local maxima, making them more prominent relative to lower ones. When combined with Gaussian smoothing, this effect is preserved for dominant peaks, while smaller, spurious local maxima are attenuated or eliminated by the smoothing operation. Specifically, we construct the following surrogate objective
\begin{equation}
\label{surrogate}
\max_{\bm{\mu}\in\mathbb{R}^d} F_{N,\sigma}(\bm{\mu}):=\mathbb{E}_{\bm{x}\sim \mathcal{D},\bm{w}\sim\mathcal{N}(\bm{\mu},\sigma^2I_d)}[e^{Nf_{\bm{w}}(\bm{x})}].
\end{equation}

During optimization, we adopt a progressive homotopy strategy. The power parameter $N$ is gradually increased so as to align the global optimum of the smoothed objective with that of the original, un-smoothed problem. Simultaneously, the smoothing scale $\sigma$ is progressively decreased to maintain sufficient curvature for fast solution improvement. The effects of these two mechanism are illustrated with a simple example in Figure \ref{N-sigma-effect}. Specifically, (a) shows that a larger $N$ aligns $\bm{w}^*=\arg\max_{\bm{w}}G(\bm{w})$ and $\bm{\mu}^*=\arg\max_{\bm{\mu}}F_{N,\sigma}(\bm{\mu})$ better, and (b) shows that a larger $\sigma$ leads to larger slopes when $\bm{\mu}$ is away from $\bm{w}^*$ and a smaller $\sigma$ leads to larger slopes when $\bm{\mu}$ is near $\bm{w}^*$.

\begin{figure}[ht]
      \begin{tabular}{cc}
	%\centering
        \subfloat[Increase $N$.]{\includegraphics[scale=0.34]{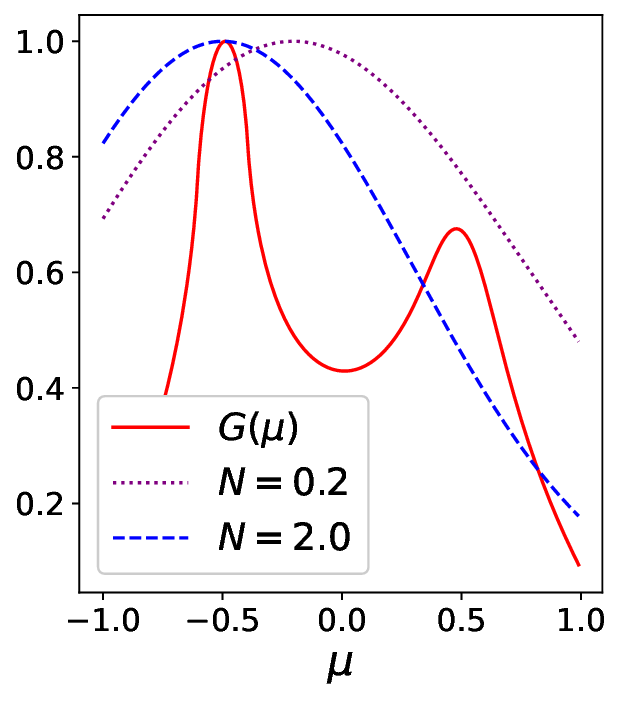} }
        &
        \subfloat[Decay $\sigma$.]{\includegraphics[scale=0.34]{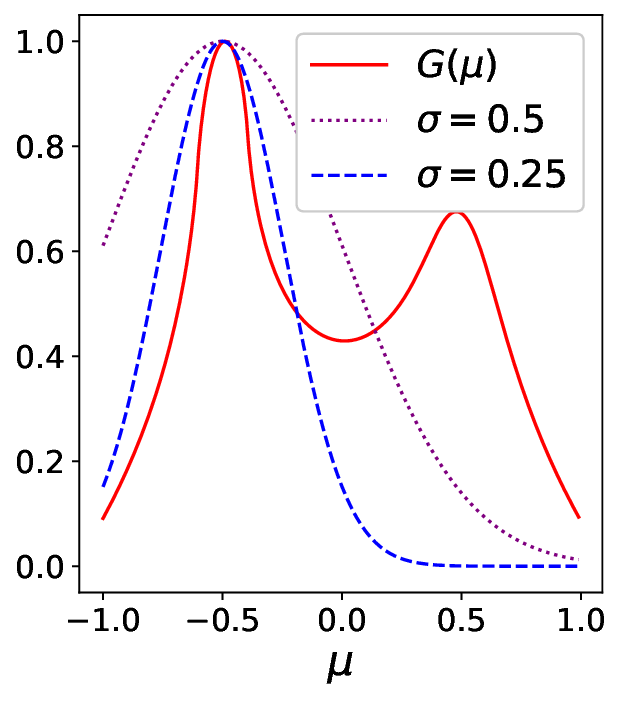} }
\end{tabular}
       	\caption{Graph of approximated $G(\mu)$ and $F_{N,\sigma}(\mu)$. (a) $N$ is varied while $\sigma\equiv 0.8$. (b) $\sigma$ is varied while $N\equiv 2.0$. Here, $G(\mu)=\mathbb{E}_{\epsilon\sim I[-0.1,0.1]}[f(\mu+\epsilon)]$, $G_N(\mu)=\mathbb{E}_{\epsilon\sim I[-0.1,0.1]}[e^{Nf(\mu+\epsilon)}]$, $F_{N,\sigma}(\mu)=\mathbb{E}_{x\sim\mathcal{N}(\mu,\sigma^2)}[G_N(\mu)]$, and $f(\mu)=-\log((\mu+0.5)^2+10^{-5})-\log((\mu-0.5)^2+10^{-2})+10$ for $|\mu|\leq 1$ and $f(\mu)=0$ for $|\mu|>1.$  All graphs are scaled to have a maximum of 1 for easier comparisons.} 
\label{N-sigma-effect}
\end{figure}

%Note that GS-PowerOpt cannot be applied to solve (\ref{stochastic-optimization}), because it requires an unbiased sample estimate of the gradient of $F(\bm{\mu}):=\mathbb{E}_{\bm{w}\sim\mathcal{N}(\bm{\mu},\sigma^2I_d)}[e^{NG(\bm{w})}]$, which is difficult to construct since only a sample estimate of $G(\bm{w})$ is available under our settings. 

\subsection{Gradients for Solving (\ref{surrogate})}
\label{gradient}
Suppose the first-order gradient of $f$ is accessible. It can be derived that:
\begin{equation}
\label{F-gradient}
\nabla_{\bm{\mu}} F_{N,\sigma}(\bm{\mu}) =  \mathbb{E}_{\bm{w}\sim \mathcal{N}(\bm{\mu}, \sigma^2I_d),\bm{x}\sim \mathcal{D}}\left[\frac{\partial e^{Nf_{\bm{w}}(\bm{x})}}{\partial \bm{w}}\right].
\end{equation}
Therefore, an un-biased sample estimate for $\nabla F_{N,\sigma}(\bm{\mu})$ is
\begin{equation}
\label{F-gradient}
\hat{\nabla} F_{N,\sigma}(\bm{\mu}) = \frac{1}{KJ} \sum_{k=1}^K\sum_{j=1}^J \left.\frac{ \partial e^{Nf_{\bm{w}}(\bm{x}_j)}}{\partial \bm{w}}\right|_{\bm{w}=\bm{w}_k},
\end{equation}
where $\{\bm{x}_j\}_{j=1}^J$ is a batch of independently independently distributed (i.i.d.) samples drawn from $\mathcal{D}$, and $\{\bm{w}_k\}_{k=1}^K$ is a batch of i.i.d. samples from $\mathcal{N}(\bm{\mu},\sigma^2I_d)$.

\subsection{Our Algorithm}
Our algorithm follows the step of stochastic gradient ascent, with an adaptive power $N$ and an adaptive $\sigma$. Specifically, the updating equations at step $t$ are as follows.
\begin{equation}
\begin{split}
\label{update-rule}
&N_t = N_{t-1} +\phi_t\Delta;\\
&\sigma_t = b + \sigma_{0} \beta^t;\\
&\bm{\mu}_{t} = \bm{\mu}_{t-1} + \alpha_t \hat{\nabla}F_{N_{t},\sigma_{t}}(\bm{\mu}_{t-1}),\\
\end{split}
\end{equation}
where $\Delta$ denotes the pre-slected total increment of $N$ in the updating process, $\beta\in (0,1)$ is the decay factor for $\sigma$, $\phi_t$ is the increasing rate for $N$, and $\alpha_t$ is the learning rate for $\bm{\mu}$. At the end of step $T$, we output $\bm{\mu}^* = \arg\max_{\bm{\mu} \in \{\bm{\mu}_1, \dots, \bm{\mu}_T\}} \hat{G}(\bm{\mu})$, where $\hat{G}(\bm{\mu}):=\sum_{j=1}^{J}f_{\bm{\mu}_t}(\mathbf{x}_j)/J$ and $\{\mathbf{x}_j\}$ is a set of validation samples drawn from $\mathcal{D}$. The complete algorithm is in Algorithm \ref{alg:Prog-PowerHP}.

\begin{algorithm}[t]
\caption{Progressive Power Homotopy (Prog-PowerHP) for Solving (\ref{stochastic-objective}).}
\label{alg:Prog-PowerHP}
\begin{algorithmic}[1]
\REQUIRE Initial power $N_0 >0$, total power increment $\Delta>0$, initial smoothing radius $\sigma_0 > 0$, the minimum smoothing radius $b>0$, decay rate $\beta \in (0,1)$, objective $f$, initial solution candidate $\bm{\mu}_0 \in \mathbb{R}^d$, $\bm{x}$ batch number $J$, $\bm{w}$ population size iteration $K$, total number $T$ of iterations, learning rates $\{\alpha_t\}_{t=1}^{T}$ and $\{\phi_t\}_{t=1}^{T}$.
\FOR{$t = 1$ to $T$}
    \STATE $N_t = N_{t-1} + \phi_t\Delta$;
    \STATE $\sigma_{t} = \sigma_0 \beta^{t}+b$;
    \STATE Sample i.i.d. $\{\bm{w}_k\}_{k=1}^K$ from $\mathcal{N}(\bm{\mu}_{t-1}, \sigma_{t}^2 I_d)$;
    \STATE Sample i.i.d. $\{\bm{x}_j\}_{j=1}^J$ from $\mathcal{D}$;
    \STATE Compute the gradient estimator:
    \[
    \hat{\nabla}F_{N,\sigma}(\bm{\mu}_{t-1}) = \frac{1}{KJ} \sum_{k=1}^K\sum_{j=1}^J \left.\frac{ \partial e^{Nf_{\bm{w}}(\bm{x}_j)}}{\partial \bm{w}}\right|_{\bm{w}=\bm{w}_k}.
    \]
    \STATE Update:
    \[
    \bm{\mu}_{t} = \bm{\mu}_{t-1} + \alpha_t  \hat{\nabla} F_{N_t,\sigma_t}(\bm{\mu}_{t-1}).
    \]
\ENDFOR
\STATE \textbf{Return} $\bm{\mu}^* = \arg\max_{\bm{\mu} \in \{\bm{\mu}_1, \dots, \bm{\mu}_T\}} \hat{G}(\bm{\mu})$, where $\hat{G}(\bm{\mu})$ is the validation performance defined below (\ref{update-rule}).
\end{algorithmic}
\end{algorithm}

\section{Convergence Analysis}
In this section, we prove that, under mild assumptions, Prog-PowerHP generates a sequence $\{\bm{\mu}_t\}$ that converges in expectation to a point within an arbitrarily small neighborhood of the globally optimal solution $\bm{w}^*$ of (\ref{stochastic-objective}), as long as the total power increment $\Delta$ is sufficiently large. The proof is divided to two parts. The first one (Section \ref{converge-to-stationary}) shows that $\{\bm{\mu}_t\}$ from (\ref{update-rule}) contains a candidate that converges in expectation to the stationary points of $F_{N_0+\Delta,b}$ (see Remark \ref{explain-converge}), while the second part (Section \ref{converge-to-neighborhood}) shows that, for any $\delta>0$, there exists a threshold $N_{\delta}$ such that for all $N_0+\Delta>N_{\delta}$, every stationary point of $F_{N_0+\Delta,b}$ lie within a $\delta$-neighborhood of $\bm{w}^*$. These results are summarized in Corollary \ref{wrap-up}.

\subsection{Convergence in Expectation to Stationary Points of $F_{N_0+\Delta,b}$}
\label{converge-to-stationary}
We derive the properties of $F_{N,\sigma}$ in Lemma \ref{lr-assumption} -- \ref{nablaF-bound}, which are used to derive a bound of $\sum_{t=1}^{T}\alpha_t \mathbb{E}[\|\nabla F_{N_{t},\sigma_{t}}(\bm{\mu}_{t-1})\|^2]$ in Theorem \ref{main-theorem}. This bound is used by Corollary \ref{conv-rate} to derive an iteration complexity for $\lim_{T\rightarrow \infty}\min_{t\in\{1,...,T\}}\mathbb{E}[\|\nabla F_{N_{t},\sigma_{t}}(\bm{\mu}_{t-1})\|^2]=0$.

\begin{assumption}
\label{lr-assumption}
The learning rates $\phi_t$ and $\alpha_t$ satisfy
\begin{enumerate}
\item $\phi_t>0$ and $\sum_{t=1}^\infty \phi_t = 1$;
\item $\alpha_t>0$, $\sum_{t=1}^\infty \alpha_t=\infty$, and $\sum_{t=1}^\infty \alpha^2<\infty$.
\end{enumerate}
\end{assumption}

\begin{assumption}
\label{f-bound}
The maximization objective $f_{\bm{w}}(\bm{x}): \mathbb{R}^d\times \mathbb{R}^p \rightarrow \mathbb{R}$ is differentiable and Lipschitz with respect to $\bm{w}$: $|f_{\bm{w}_1}(\bm{x})-f_{\bm{w}_2}(\bm{x})|\leq L_{0f}\|\bm{w}_1-\bm{w}_2\|$ for any $\bm{w}_1$, $\bm{w}_2\in\mathbb{R}^d$, and $\bm{x}\in\mathbb{R}^p$, and has an upper bound $M_f$. Here, $L_{0f}>0$ denotes the Lipschitz constant.
\end{assumption}

\begin{assumption}
\label{Lipschitz-smooth}
The maximization objective $f_{\bm{w}}(\bm{x})$ is Lipschitz smooth in terms of $\bm{w}$. That is, there exists a constant $L_{1f}>0$ such that $\|\nabla_{\bm{w}} f_{\bm{w}_1}(\bm{x})-\nabla_{\bm{w}}f_{\bm{w}_2}(\bm{x})\| \leq L_{1f} \|\bm{w}_1-\bm{w}_2\|$ for any $\bm{w}_1, \bm{w}_2\in\mathbb{R}^d$, and any $\bm{x}$ that belongs to the sample space determined by $\mathcal{D}$.
\end{assumption}
The assumption on the learning rate $\alpha_t$ is the classic Robbins–Monro stochastic approximation conditions (e.g., \cite[Eq. (2.2.14)]{benveniste2012}), and it appears in a wide range of convergence analyses for stochastic gradient–based algorithms. For example, \cite{WangGurbuzbalaban2021} considers learning rates satisfying $\lim_{t\rightarrow \infty} \alpha_t/t^{-\rho} \in (0,\infty)$ with $\rho\in (0,1),$ which satisfies the Robbins–Monro conditions if $\rho\in(1/2,1)$. Similarly, \cite{mertikopoulos2020} established almost-sure convergence results in non-convex settings under the same step-size conditions.

The $f$-boundedness, $\bm{x}$-uniform Lipschitz continuity, and $\bm{x}$-uniform Lipschitz smoothness assumptions stated in Assumption~\ref{f-bound} and Assumption~\ref{Lipschitz-smooth} are also common in optimization for machine learning. Such regularity conditions are routinely adopted to control the sensitivity of the objective function with respect to the model parameters and to enable stability analyses. Representative examples include \cite{Iwakiri2022}, \cite{Aydore2019}, and \cite{Cutkosky2019}. In fact, Lipschitz continuity and Lipschitz smoothness are standard assumptions in theoretical works on non-convex optimization for machine learning models (\cite{HuWu2023,liQian2023}).

\begin{lemma}
\label{well-defined}
Under Assumption \ref{f-bound}, given any $N>0$ and $\sigma>0$, both $F_{N,\sigma}(\bm{\mu})$ and $\nabla_{\bm{\mu}}F_{N,\sigma}(\bm{\mu})$ are well-defined for all $\bm{\mu}\in\mathbb{R}^d$.
\end{lemma}

\begin{lemma}
\label{F-smooth}
Under Assumption \ref{f-bound} and \ref{Lipschitz-smooth}, for any $\bm{\mu}_1$, $\bm{\mu}_2\in\mathbb{R}^d$, $\| \nabla F_{N,\sigma}(\bm{\mu}_1) -  \nabla F_{N,\sigma}(\bm{\mu}_2) \|\leq L_F\|\bm{\mu}_1-\bm{\mu}_2\|$, where $L_F=e^{NM_f}N(NL_{0f}^2+L_{1f})$.
\end{lemma}

\begin{lemma}
\label{nablaF-bound}
Under Assumption \ref{f-bound}, $\| \hat{\nabla} F_{N,\sigma}(\bm{\mu}) \|^2 \leq G=N^2e^{2NM_f} L_{0f}^2$.
\end{lemma}

\begin{theorem}
\label{main-theorem}
Let $\{\bm{\mu}_t\}_{t=0}^T$ be the series produced by (\ref{update-rule}). Under Assumption \ref{f-bound} and \ref{Lipschitz-smooth}, we have 
\begin{equation*}
\begin{split}
\sum_{t=1}^{T}\alpha_t \mathbb{E}[\|\nabla &F_{N_{t},\sigma_{t}}(\bm{\mu}_{t-1})\|^2] \leq e^{(N_0+\Delta)M_f}+ H_1 \sum_{t=0}^{T-1}\alpha_t^2 \\
&-F_{N_{0},b+\sigma_{0}}(\bm{\mu}_{0})  +H_2d \sum_{t=0}^{T-1}\beta^t+H_3 \sum_{t=1}^{T} \phi_{t},
\end{split}
\end{equation*}
where $H_1=(N_0+\Delta)^3 e^{3(N_0+\Delta)M_f}((N_0+\Delta)L^2_{0f}+L_{1f})L_{0f}^2$, $H_2 = e^{(N_0+\Delta)M_f}(N_0+\Delta)L_{0f}\sigma_0(1-\beta)\sqrt{2}$, $H_3:=\max\{ (N_0e)^{-1},  e^{(N_0+\Delta)M_f} M_f \}\Delta$.
\end{theorem}

\begin{corollary}
\label{conv-rate}
\textbf{Corollary \ref{conv-rate}}. Let $\{\bm{\mu}_t\}_{t=1}^T$ be the series produced by (\ref{update-rule}) with a pre-selected deterministic $\bm{\mu}_0$ and a learning rate $\alpha_t=t^{-(1/2+\gamma)}$ where $\gamma\in (0,1/2)$ and $T\geq 2$. Under Assumption \ref{lr-assumption}, \ref{f-bound}, and \ref{Lipschitz-smooth}, for any $\varepsilon>0$, whenever $T> (\frac{C_{0,\Delta}(1-2\gamma)}{2-2^{1/2+\gamma}}d\varepsilon^{-1})^{2/(1-2\gamma)}=O_{\Delta}((d \varepsilon^{-1})^{2/(1-2\gamma)})$, we have that
$$ \min_{t\in\{\lfloor T/2\rfloor,...,T\}} \mathbb{E}[\|\nabla F_{N_{t},\sigma_{t}}(\bm{\mu}_{t-1})\|^2] \leq \varepsilon. $$ Here, $C_{0,\Delta}:=e^{(N_0+\Delta)M_f} - F_{N_{0},b+\sigma_{0}}(\bm{\mu}_{0}) + H_1 \sum_{t=1}^{\infty}t^{-(1/2+\gamma)} +H_2(1-\beta)^{-1}+H_3$, $C_{1,\Delta}= \max\{1, 2/C_{0,\Delta}\}$, $\lfloor T/2 \rfloor$ denotes the largest integer no greater than $T/2$, and $\{H_1,H_2,H_3\}$ are as those defined in Theorem \ref{main-theorem}.
\end{corollary}

\begin{corollary}
\label{Delta-irrelevant}
The dependence of the $O$-term on $\Delta$ can be removed if we assume that $M_f\leq 0$.
\end{corollary}
\begin{remark}
\label{explain-converge}
In Corollary \ref{conv-rate}, define 
$$T^*:= \min_{\tau\in\{ \lfloor T/2\rfloor,...,T\}}\mathbb{E}[\|\nabla F_{N_{t},\sigma_{t}}(\bm{\mu}_{t-1})\|^2]$$
and $\bm{\mu}^*_T:=\bm{\mu}_{T^*}$. As $T\rightarrow +\infty$, for all $t\in [\lfloor T/2\rfloor,T]$, $N_t$ and $\sigma_t$ satisfy $N_t\rightarrow N_0+\Delta$ and $\sigma_t\rightarrow b$. Under the assumed Lipschitz continuity, Lipschitiz smoothness, and $boundedness$ of $f$, it follows that $\nabla F_{N_t,\sigma_t}$ converges uniformly to $\nabla F_{N_0+\Delta,b}$ over the interval $[\lfloor T/2\rfloor,T]$ as $T$ becomes sufficiently large. Consequently, $\bm{\mu}^*_T$ converges in expectation to a stationary point of $F_{N_0+\Delta,b}$.
\end{remark}

\subsection{Convergence to the Neighborhood of Global Optimum}
\label{converge-to-neighborhood}
The main task in this section is to show in Theorem \ref{main-convergence-thm} that, if $N$ is greater than some threshold, all the stationary points of the surrogate $F_{N,\sigma}(\bm{\mu})$ defined in (\ref{surrogate}) lie in a small neighborhood of $\bm{w}^*:=\arg\max_{\bm{w}}G(\bm{w})$. 

\begin{assumption}
\label{ess}
For any fixed $\bm{w}$, define $\Psi(\bm{w}):=\inf\{a\in\mathbb{R} | \mathbb{P}(f_{\bm{w}}(\bm{x})\leq a)=1\}$, where $\mathbb{P}$ is regarding the distribution $\mathcal{D}$ of $\bm{x}$. We assume that (1) $\bm{w}^*:=\arg\max_{\bm{w}}G(\bm{w})$ is also the unique global maximizer of $\Psi$; (2) for any $\delta>0$, $\sup_{\|\bm{w}-\bm{w}^*\|\geq\delta}\Psi(\bm{w})<\Psi(\bm{w}^*)$; and (3) for any $\epsilon>0$, there exist $c_{\epsilon}>0$ and $\delta_{\epsilon}>0$ such that $\inf_{\|\bm{w}-\bm{w}^*\|\leq \delta_{\epsilon}}\mathbb{P}(f_{\bm{w}}(\bm{x})\geq \Psi(\bm{w})-\epsilon)>c_{\epsilon}$.
\end{assumption}
\begin{remark}
$\Psi(\bm{w})$ measures the best-case fitness $f_{\bm{w}}(\bm{x})$, achievable under the data distribution $\bm{x}\sim\mathcal{D}$. It is also called as the essential supreme. We emphasize that the exact coincidence required in Part (1) can be relaxed to an approximate one. Since our algorithm converges to a neighborhood of $\arg\max_{\bm{w}}{\Psi(\bm{w})}$ (see Corollary \ref{wrap-up}), convergence to a neighborhood of $\bm{w}^*$ is also sufficient for our purposes.
\end{remark}

Intuitively, Part (1) assumes an alignment between average fitness and best-case fitness, namely that a parameter value $\bm{w}$ which maximizes the expected fitness $\mathbb{E}_{\bm{x}\sim\mathcal{D}}[f_{\bm{w}}(\bm{x})]$ also attains near-optimal best-case performance, measured by the essential supremum of $f_{\bm{w}}(\bm{x})$ under $\bm{x}\sim\mathcal{D}$. This assumption rules out pathological situations in which a model achieves a very high fitness on a vanishingly small subset of inputs while performing poorly on the bulk of the data.

Such an alignment is reasonable when the fitness $f_{\bm{w}}(\bm{x})$ varies smoothly over the data domain and its distribution under $\mathcal{D}$ is well behaved. For instance, if $f_{\bm{w}}(\bm{x})$ is uniformly Lipschitz in $\bm{x}$ and the data distribution assigns non-negligible probability mass to neighborhoods in the input space, then high fitness values cannot be achieved on isolated points without also improving the fitness on a set of inputs with positive probability, which in turn increases the expectation. Similarly, if the random variable $f_{\bm{w}}(\bm{x})$ has light tails (e.g., is sub-Gaussian), then extreme values are tightly controlled by the mean, preventing sharp discrepancies between best-case and average performance.

This type of average-best alignment appears in intensively studied subfields of machine learning. In partial-label learning\footnote{More ongoing research works on partial labelling can be found in \cite{GongBisht2025,WangWu2025,lv2024makes}.}, for example, \cite{cabannnes2020} show in their Theorem 1 that, under appropriate structural assumptions on label ambiguity, minimizing the expected true loss coincides with minimizing an infimum-based loss that captures the best achievable performance over admissible labels. While problem-specific, this result illustrates that, in structured learning settings, optimizing an average criterion can be equivalent to optimizing a best-case criterion.

More broadly, regularity assumptions such as Lipschitz continuity and light-tailed distributions are standard in optimization and generalization analyses in machine learning (e.g., \cite{sabato2013,zeng2022}). Under such conditions, it is natural to expect that models achieving superior average fitness also exhibit superior best-case behavior.

Part (3) in Assumption \ref{ess} assumes that, for each $\bm{x}$, the near-optimal values of $f_{\bm{w}}(\bm{x})$ are supported by a non-negligible fraction of the data distribution $\mathcal{D}$, ruling out pathological cases where optimal fitness is achieved only on sets of vanishing probability. This belongs to the kind of assumptions that assume lower bounds on the probability of near-optimal sets, which are common among research works in machine learning (e.g., \cite{tong2020neyman, perdomo2020performative}).

\begin{lemma}
\label{sn-lemma}
Under Assumption \ref{f-bound} and \ref{ess}, for any $\delta>0$, we have
(1) $\lim_{N\rightarrow +\infty}s_N/d_N=0$, where $s_N:=\sup_{\|\bm{w}-\bm{w}^*\|\geq\delta}G_N(\bm{w})$, $d_N:=G_N(\bm{w}^*)e^{-N\eta/4}$, and $\eta:= \Psi(\bm{w}^*)-\sup_{\|\bm{w}-\bm{w}^*\|\geq \delta} \Psi(\bm{w})$, and (2) there exists $\delta'\in(0,\delta)$ and $N_{\delta'}$ such that $G_N(\bm{w})>d_N$ whenever $\|\bm{w}-\bm{w}^*\|<\delta'$ and $N>N_{\delta'}$.
\end{lemma}

\begin{theorem}
\label{main-convergence-thm}
Suppose Assumption \ref{f-bound} and \ref{ess} hold. Denote $\bm{\mu}=[\mu_1,...,\mu_d]$ and $\bm{w}^*=[w_1^*,...,w_d^*]$. For any $\delta>0$, $\sigma>0$, and $M>0$, there exists a threshold $N_{\delta,\sigma,M}>0$ such that, whenever $N>N_{\delta,\sigma,M}$ and $\|\bm{\mu}\|\leq M$, for all $i\in\{1,2,...,d\}$ we have $\frac{\partial F_{N,\sigma}(\bm{\mu})}{\partial \mu_i}<0$ if $\mu_i>w_i^*+\delta$, and $\frac{\partial F_{N,\sigma}(\bm{\mu})}{\partial \mu_i}>0$ if $\mu_i<w_i^*-\delta$. 
\end{theorem}

\begin{assumption}
\label{bounded-solutions}
Let $\bm{\mu}_0$ be fixed. We assume that there exists a threshold $\Delta_0>0$ such that the sequence $\{\bm{\mu}_t\}$ produced by the iteration rule (\ref{update-rule}) lies in a bounded region $\|\bm{\mu}_t\|<M_{\bm{\mu}_0}$ whenever $\Delta>\Delta_0$. 
\end{assumption}
\begin{remark}
This is a reasonable assumption because, for any $\delta,\sigma_t>0$, the expectation of the updating term $\alpha_t\hat{\nabla}F_{N_t,\sigma_t}(\bm{\mu}_{t-1})$ effectively pulls $\bm{\mu}_t$ toward a neighborhood of $\bm{w}^*$ when $N_t$ is sufficiently large.
\end{remark}
Below, we summarize our theoretical results.
\begin{corollary}
\label{wrap-up}
Suppose Assumption \ref{lr-assumption}---\ref{bounded-solutions} hold. Let $\delta,b,N_0>0$ be given, along with an initial point $\bm{\mu}_0\in\mathbb{R}^d$. Let $\Delta_0$ and $N_{\delta,b,M_{\bm{\mu}_0}}$ be the thresholds specified in Assumption \ref{bounded-solutions} and Theorem \ref{main-convergence-thm}, respectively. Then, if $\Delta>\max\{\Delta_0,  N_{\delta,b,M_{\bm{\mu}_0}}\}$, $\bm{\mu}^*_t$ defined in Remark \ref{explain-converge} converges in expectation to $S_{\bm{w}^*,\delta}:=\{\bm{w}\in\mathbb{R}^d | |w_i-w_i^*|<\delta, i\in\{1,2,...,d\}\}$ as $t\rightarrow \infty$. The iteration complexity is $O_{\Delta}((d\varepsilon^{-1})^{2/(1-2\gamma)})$ if we set the learning rate $\alpha_t=(t+1)^{-(1/2+\gamma)}$, where $\gamma\in (0,1/2)$. The dependence of $O_{\Delta}$ on $\Delta$ can be removed if we further assume that $M_f\leq 0$.
%Here, $w_i$ and $w_i^*$ denote the $i^{th}$ entry of $\bm{w}$ and $\bm{w}^*$, respectively. 
\end{corollary}
\begin{proof}
Corollary \ref{conv-rate} and Remark \ref{explain-converge} imply that $\bm{\mu}^*_t$ converges to a stationary point of $F_{N_0+\Delta,b}$. Theorem \ref{main-convergence-thm} implies that any stationary point approached by $\bm{\mu}^*_t$ lie in the neighborhood $S_{\bm{w}^*,\delta}$. Corollary \ref{Delta-irrelevant} shows that the dependence of the iteration complexity on $\Delta$ can be removed if $M_f\leq0$.
\end{proof}

\section{Experiment}
In this section, we evaluate our methods and focus on synthetic benchmarks—specifically phase retrieval and two-layer ReLU networks training—where the landscape geometry is well-characterized by existing theoretical literature (\cite{pmlr-safran18a, bandeira2014, Mondelli2018,ma2021spectral}). Standard benchmarks like MNIST or CIFAR-10 often obscure optimization failure modes due to the inherent redundancies of high-dimensional over-parameterization. By targeting limited-sample regimes in phase retrieval and under-parameterized settings in neural network training, we subject our algorithm to a rigorous stress test within landscapes known to harbor spurious local minima. This controlled setting more rigorously tests the effectiveness of power transformation and Gaussian smoothing by isolating their optimization behavior — something large-scale, generalization-oriented datasets cannot do.

We compare our performance with several widely used first-order optimizers, including the mini-batch SGD (\cite{bottou2010large}), SGD with Momentum (\cite{SGDMomentum2013}), Adam (\cite{kingma2017adam}), AdamW (\cite{Loshchilov2017}), and Sharpness-Aware Minimization (SAM, \cite{foret2021SAM}). Additionally, we include the homotopy-based method of SLGHr (\cite{Iwakiri2022}) as a benchmark for ablation study.

\begin{comment}
\subsection{Maximize Expectation of 2logs}
In this experiment, the objective is
\begin{equation}
\label{E2logs-obj}
\max_{\bm{w}}\mathbb{E}[f(\bm{w},\bm{x}_1,\bm{x}_2)], 
\end{equation}
where the expectation is with respect to $\bm{x}_1\sim\mathcal{N}(\bm{m}_1,0.25I_d)$ and $\bm{x}_2\sim\mathcal{N}(\bm{m}_2,0.25I_d)$, with all entries of $\bm{m}_1$ being $-0.5$, all entries of $\bm{m}_2$ being $0.5$, and
\begin{equation}
\begin{split}
f(\bm{w},\bm{x}_1,\bm{x}_2):=&-\ln\left(\|\bm{w-x}_1\|^2+10^{-5} \right)\\
&-\ln\left(\|\bm{w-x}_2\|^2+10^{-2} \right).
\end{split}
\end{equation}
\end{comment}

\subsection{Phase Retrieval}
Given the unseen phase $\bm{x}_0\in\mathbb{R}^d$ and measurements of the form $y=\langle \bm{a}, \bm{x}_0 \rangle$, the phase retrieval (PR) problem seeks to recover $\bm{x}_0\in\mathbb{R}^d$ from these observations. Beyond this classical formulation, modern phase retrieval research has increasingly incorporated sophisticated artificial intelligence paradigms, including deep generative priors (\cite{Hand2018}), unsupervised teacher-student distillation protocols (\cite{quan2023unsupervised}), and conditional generative models (\cite{Uelwer2023}).

The classical phase retrieval problem is typically formulated as the following optimization task (e.g., \cite{ma2021spectral}), which is notorious for exhibiting multiple spurious local minima:
\begin{equation}
\label{PR-objective}
\min_{\bm{x}\in\mathbb{R}^d} \mathcal{L}(\bm{x}):= \mathbb{E}_{\bm{a}\sim\mathcal{N}(\bm{0},I_d)} \left[( \langle \bm{a}, \bm{x} \rangle^2-\langle\bm{a},\bm{x}_0\rangle^2 )^2\right].
\end{equation}
In practice, a batch of observed signals, $\{\bm{a}_n\}_{n=1}^N$ and $\{y_n\}_{n=1}^N$ with $y_n = \langle \bm{a}_n, \bm{x} \rangle$, are usually available. The sample-based objective is constructed as
\begin{equation}
\label{PR-objective-est}
\min_{\bm{x}\in\mathbb{R}^d} \hat{\mathcal{L}}(\bm{x}):=\frac{1}{N}\sum_{n=1}^N ( \langle \bm{a}_n, \bm{x} \rangle^2-y_n^2 )^2.
\end{equation}
In our simulated experiments, we use the following standard metric (\cite{Hand2018,zhang2017nonconvex}) for measuring the quality of a given solution $\bm{x}$:
\begin{equation}
\label{PR-metric}
\mathcal{M}(\bm{x}):= \min(\|\bm{x}-\bm{x}_0\|_2, \|\bm{x}+\bm{x}_0\|_2)/\|\bm{x}_0\|_2,
\end{equation}
where $\|\cdot\|_2$ denotes the $L_2$-norm.

Under each set of selected values of $d$ and $N$, we perform 100 experiments. In each experiment, we randomly generate a true solution $\bm{x}_0$ from $\mathcal{N}(\bm{0},0.1I_d)$, an initial solution candidate $\bm{x}^{(0)}$ from $\mathcal{N}(\bm{0},I_d/\sqrt{d})$, and $N$ independent samples $\{\bm{a}_n\}_{n=1}^N$ from $\mathcal{N}(\bm{0},I_d)$. Then, with these values, we apply each compared optimization method to solve (\ref{PR-objective-est}), where each solution update uses a mini-batch of $\{\bm{a}_n\}_{n=1}^N$. Let $\{\bm{x}^{(t)}\}_{t=1}^T$ denote the solution candidates produced by the performed.  If $\min_{t} \mathcal{M}(\bm{x}^{(t)})\leq 0.001$, we consider that the true solution $\bm{x}_0$ is recovered and call the method to be successful for this experiment. Define $\bm{x}_i^*=\arg\min_{t}\mathcal{M}(\bm{x}^{(t)})$, where $i$ denotes the experiment index. Average results obtained over the 100 experiments are reported in Table \ref{PR-table1}, where we also report the two-sided $t$-test result to check if our algorithm produces a significantly lower $\mathcal{M}(\bm{x}^*)$ than other algorithms. 

\begin{table}[htb!]
\caption{Performances on Minimizing (\ref{PR-objective-est}). The metric $\mathcal{M}$ is defined in (\ref{PR-metric}) and $\bar{\mathcal{M}}$ denotes the average. SR refers to the success rate out of the 100 experiments, $t$ refers to the statistic of the two-sided $t$-test on the null hypothesis $H_0:\;\mathcal{M}(\bm{x}^*)-\mathcal{M}(\bm{y}^*)=0$, where $\bm{x}^*$ is the optimal solution produced by our algorithm, and $\bm{y}^*$ is produced by the compared algorithm. Time refers to the average number of iterations taken to the best solution. The iteration number $T$ and sample number $N$ are fixed at $60,000$ and $400$, respectively. Under the most difficult case where $d=200$, we relax the threshold for defining a success from $0.001$ to $0.05$.}
\label{PR-table1}
\centering%
\begin{tabular}{ p{0.35cm} | p{2.0cm}  p{0.55cm} p{0.5cm}  p{0.75cm} p{0.75cm}   }
\midrule
$d$ & Algorithm  & $\bar{\mathcal{M}}$ & SR. & \;\;\;$t$ & Time \\
\toprule
\multirow{4}{*}{$100$} & \textbf{Our Algo}.  &  0.07 &  $92\%$   & \;\; - & $56,397$ \\
                                 & SGD    & 0.51   &  $54\%$   &   $-7$      & $16,257$ \\  
                                 & SLGHr & 0.80   &  $0\%$     &   $-23$    & $52,849$ \\
                                 & SAM    &  0.68  &   $41\%$  &    $-9$     &  $24,524$ \\ 
                                & SGD-Mom  &  0.49  &  $56\%$    & $-7$ & $21,572$  \\  
                                & AdamW    &  0.60 &  $47\%$   &  $-8$       & $14,653$ \\
                                & Adam    &  0.65 &  $43\%$   &     $-9$    & $10,869$ \\
                                
\hline
\multirow{4}{*}{$150$} & \textbf{Our Algo}.  &  0.23 &  $73\%$   & \;\; - & $47,974$ \\
                                & SLGHr                    &  0.82 &  $0\%$     &$-13$ & $45,047$ \\
                               &  SAM                       &  1.08  &  $0\%$    &$-19$  & $2,812$  \\
                               & SGD-Mom               &  1.11  &  $2\%$    & $-20$ & $1,515$  \\  
                               & AdamW                   &  1.10 &  $4\%$   &  $-19$     & $4,629$ \\
                                & Adam                     & 1.08   &  $5\%$   &  $-18$     & $1,922$ \\
                                & SGD                      & 1.07   &  $6\%$   &    $-18$   & $4,460 $ \\
\hline
\multirow{7}{*}{$200$} & \textbf{Our Algo}.  &  0.67 &  $23\%$   & \;\; - & $18,175$ \\
                                & SLGHr    &  1.02 &  $0\%$   &  $-9$     & $21,942$ \\
                                & SAM    &  1.14 &  $0\%$   &  $-12$     & $676$ \\
                                & SGD-Mom    &  1.14 &  $0\%$   &  $-12$     & $91$ \\
                                & AdamW    & 1.15   &  $ 0\%$   &  $-12$     & $1,507$ \\
                                & Adam    & 1.15   &  $ 0\%$   &  $-12$     & $1,518$ \\
                                & SGD    & 1.13   &  $0\%$   &    $-12$   & $84 $ \\
\bottomrule
\end{tabular}
\end{table}

It is well-known that the PR problem exhibits a transition into difficulty when the measurement-to-dimension ratio $N/d$ decreases. In particular, prior work shows that the recovery is impossible when $N/d < 1$, while above this threshold spectral and first-order methods begin to succeed \cite{Mondelli2018}. However, in the regime where $N/d\in [2,4]$, recovery is theoretically possible but empirically challenging, as the non-convex loss landscape is poorly conditioned \cite{bandeira2014,ma2021spectral}. Motivated by these observations, we focus our experiments on this information-scarce regime by fixing $N=400$ and varying $d$ so that $N/d \in [2,4]$, where algorithmic performance is most sensitive to optimization strategy.

The experiment results in Table \ref{PR-table1} show that our algorithm outperforms all the compared algorithms. The $t$ statistics indicate that the results are statistically significant. 

Our results align with the established theory of phase retrieval: when $N/d$ is small (e.g, $N/d\in [2,4)$), the optimization geometry is fragile and recovery is sensitive to the choice of algorithm. Beyond this regime, the non-convex loss landscape becomes benign—characterized by the absence of spurious local minima and the presence of negative curvature at all saddle points (\cite{sun2018geometric}), and hence, the performance differences between various optimization methods largely vanish.

\subsection{Two-layer Neural Network Training}
We evaluate our algorithm using the experimental setup for training two-layer neural networks (without bias terms) as described in \cite{pmlr-safran18a}. Specifically, given a set of orthonormal ground-truth weight vectors $\{\bm{v}_i\}_{i=1}^k$, the objective is to solve $\min_{\bm{W}\in\mathbb{R}^{d\times n}} \mathcal{L}(\bm{W})$, where
\begin{equation}
\label{twolayer-loss}
\mathcal{L}(\bm{W})=\mathbb{E}_{\bm{x}\sim \mathcal{N}(\bm{0},I_k)}\left[\frac{1}{2}\left(\sum_{j=1}^n[\bm{w}_j'\bm{x}]_+ - \sum_{i=1}^k[\bm{v}_i'\bm{x}]_+\right)^2\right],
\end{equation} 
$\bm{W}:=[\bm{w}_1,...,\bm{w}_n]$ represents the learnable weights and $[z]_+=\max(0,z)$ is the ReLU activation. \cite{pmlr-safran18a} demonstrate that locating the global optimum is particularly challenging in the exactly-parameterized case $k=n\in[16,20]$, especially for $n\in[16,20]$, due to the prevalence of spurious local minima. While over-parameterization ($n>k$) is known to significantly reduce this difficulty, we choose to address the more challenging regimes. We evaluate our algorithms across four settings: the exactly-parameterized cases where $k=n=20$ and $k=n=25$, as well as the under-parameterized cases where $(k,n)=(20,19)$ and $(k,n)=(25,24)$.

With each pair of selected values for $(k,n)$, we perform 100 experiments on each of the compared optimization algorithms. In each experiment, we randomly generate a set $\{\bm{v}_i\}_{i=1}^k$ of orthonormal vectors in $\mathbb{R}^k$, and sample a batch of 30 $\bm{x}$'s from $\mathcal{N}(\bm{0},I_k)$ for each update of the algorithm's solution. All the algorithms share the same initial solution candidate. Following \cite{pmlr-safran18a}, we use the test error as the evaluation metric, which is computed with a set of 5,000 test samples of $\bm{x}$:
\begin{equation}
\label{test-error}
e(\bm{W}) = \frac{1}{5000} \sum_{l=1}^{5000}\frac{1}{2} \left(\sum_{j=1}^n[\bm{w}_j'\bm{x}_l]_+ - \sum_{i=1}^k[\bm{v}_i'\bm{x}_l]_+\right)^2.
\end{equation}
For each algorithm, we report statistics about the optimal solution $\bm{W}^*:=\arg\min_{\tau=1}^T e(\bm{W}_\tau)$ found in the solution-update process. The pre-selected total number $T$ of iterations is set as $35,000$ each method (we choose a population size of 1 for our method). Again, we consider an algorithm successful (i.e., the global minimum is found) if $e(\bm{W}^*)\leq 0.001$. 

\begin{table}[htb!]
\caption{Performances on Minimizing (\ref{twolayer-loss}). 100 experiments are performed for each algorithm. $e$ represents the average of $\{e(\bm{W}^*_i)\}_{i=1}^{100}$, where $\bm{W}^*_i$ represents the optimal solution found in the $i^{th}$ experiment, and the metric $e(\bm{W})$ is defined in (\ref{test-error}). SR refers to the success rate out of the 100 experiments, $t$ refers to the statistic of the two-sided $t$-test on $e(\bm{W}^*)<e(\bm{W}_c^*)$, where $\bm{W}^*$ is produced by our algorithm and $\bm{W}^*_c$ is produced by the compared algorithm.}

\label{TwoLayer-table1}
\centering%
\begin{tabular}{ p{0.3cm}  p{0.45cm}| p{2.0cm}  p{0.4cm} p{0.7cm}  p{0.70cm} p{0.75cm}   }
\midrule
$n$ &$k$ & Algorithm  & SR. & $e$  & \;\;\;$t$ & Time \\
\toprule
\multirow{4}{*}{$18$} &\multirow{4}{*}{$19$}& \textbf{Our Algo}. &  $0\%$ &  $\mathbf{0.058}$    & -\;\; & $26,587$ \\
                                &                                   & SLGHr  &  $0\%$  &   0.073    &   $-8 $     & $32,212 $ \\
                                &                                   & SGD    &  $0\%$   &  0.073    &  $-11$      & $ 32,452$ \\
                                &                                   & SGD-Mom  &  $0\%$  &  0.074    &   $-11$     & $23,745$ \\
                                &                                   & SAM  &  $0\%$  &  0.078    &   $-17$     & $24,902$ \\
                                &                                   & AdamW  &  $0\%$   &  0.079   &   $-10$     & $32,050$ \\
                                &                                   & Adam     &  $0\%$  &  0.081    &   $-11$     & $32,707$ \\
\hline
\multirow{4}{*}{$19$} &\multirow{4}{*}{$20$}& \textbf{Our Algo}. &  $0\%$ &  $\mathbf{0.061}$   & \;\; - & $28,175$ \\
                                &                                   & SLGHr  &  $0\%$  &  0.077    &   $-8$     & $31,725$ \\
                                &                                   & SGD-Mom  &  $0\%$  &  0.081    &   $-9$     & $24,217$ \\
                                &                                   & SAM  &  $0\%$  &  0.083    &   $-11$     & $32,647$ \\
                                &                                   & AdamW  &  $0\%$  &  0.085   &   $-9$     & $32,929$ \\
                                &                                   & Adam  &  $0\%$   &  0.082   &   $-8$     & $32,978$ \\
                                &                                   & SGD   &  $0\%$  &  0.078     &  $-9$      & $32,036$ \\
\hline
\multirow{4}{*}{$19$} &\multirow{4}{*}{$19$}& SLGHr  &  $10\%$  &  0.026     &   $1.4 $     & $32,554 $ \\
                                &                                   & \textbf{Our Algo}. &  $7\%$  &  $\mathbf{0.029}$    & \;\; - & $30,481$ \\
                                &                                   & SGD  &  $11\%$  &  0.029     &  $-0.1$      & $31,327$ \\
                                &                                   & AdamW  &  $6\%$  &  0.031    &   $-1.0$     & $30,176$ \\
                                &                                   & Adam   &  $3\%$  &  0.032   &   $-1.3$     & $30,437$ \\
                                &                                   & SAM  &  $7\%$  &  0.032    &   $-1.2$     & $31,664$ \\
                                &                                   & SGD-Mom  &  $5\%$  & 0.035    &   $-2.3$     & $33,852$ \\
\hline
\multirow{4}{*}{$20$} &\multirow{4}{*}{$20$}& SLGHr  &  $10\%$  &  0.027     &   $1.2 $     & $32,127 $ \\
                                &                                   & \textbf{Our Algo}. &  $4\%$  &  $\mathbf{0.030}$    & \;\; - & $27,972$ \\
                                &                                   & Adam   &  $3\%$  &  0.033   &   $-1.0$     & $30,598$ \\
                                &                                   & SAM  &  $5\%$  &  0.033    &   $-1.1$     & $31,606$ \\
                                &                                   & AdamW  &  $6\%$  &  0.035    &   $-1.8$     & $31,294$ \\
                                &                                   & SGD  &  $2\%$  &  0.035     &  $-1.8$      & $32,479$ \\
                                &                                   & SGD-Mom  &  $6\%$  & 0.037    &   $-2.7$     & $21,079$ \\

\bottomrule
\end{tabular}
\footnotetext{Total number of generations: 1000. $\bm{m}_1=[-.5,-.5]$. $\bm{m}_2=[.5,.5]$.}
\footnotetext{\textbf{Notation.} $\mathcal{P}_0$ denotes the initial population. For any real random vector $\bm{z}$ and any two real scalar $a<b$, $\bm{z} \sim$ uniform$[a, b]$ denotes that each entry of $\bm{z}$ is sampled uniformly from the interval $[a, b]$.}
\end{table}

The results are summarized in Table \ref{TwoLayer-table1}. While our algorithm exhibits comparable or slightly improved performance relative to most baseline methods in the exactly-parameterized regime ($n=k$), these gains are generally not statistically significant. In contrast, in under-parameterized settings ($n<k$), our method demonstrates a clear and consistent advantage, achieving substantially lower loss values with statistically significant margins. These results highlight the superior landscape navigation capabilities of our approach; specifically, our algorithm consistently identifies deeper local minima in scenarios where capacity constraints preclude a zero-loss solution. This empirical behavior aligns with our theoretical framework, which suggests that the power transformation effectively reweights the optimization landscape to favor more optimal local basins.

\subsection{Explanation of Results}
We hypothesize that the consistent outperformance of our approach stems from a dual-mechanism synergy between the power transform and landscape smoothing. Specifically, the power transformation reshapes the remaining topography to reweight the 'attractiveness' of various basins, allowing the iterates to prioritize deeper, more optimal valleys. Complementing this, the smoothing component serves to regularize the optimization surface by effectively marginalizing shallow, spurious local minima that typically trap first-order methods in data-sparse or under-parameterized regimes. Together, these mechanisms enable our algorithm to bypass high-loss noise and converge toward superior local basins that are often inaccessible to standard optimizers. Interestingly, the convergence to identical success rates in the exactly-parameterized case ($n=k$) suggests that while this dual strategy is highly proficient at navigating 'cluttered' environments, it still encounters the fundamental symmetry-breaking hurdles inherent in perfect recovery tasks. Ultimately, these findings position our method as a robust solution for non-convex problems where model capacity is constrained or sampling is limited.

\section{Conclusion}
In this work, we propose a novel global optimization framework that integrates a progressive power-transformation mechanism within a homotopy scheme, and we establish its convergence to the global optimum with provable iteration complexity. We evaluate the proposed method on constrained nonconvex landscapes arising from limited data availability—such as phase retrieval in the regime ($n/d \in [2,4]$)—and limited model capacity, as in under-parameterized neural networks with ($n < k$).

In the under-parameterized neural network setting, where exact recovery is mathematically impossible, our method consistently converges to solutions with significantly lower loss than other first-order optimizers. A similar advantage is observed in phase retrieval, particularly near the information-theoretic limit (n/d = 2) \cite{wang2019generalized}. Together, these results suggest that the proposed combination of power transformation and Gaussian smoothing provides an effective mechanism for navigating degenerate non-convex landscapes, with potential implications for reducing model size and computational demands in large-scale deep learning systems.

\bibliography{Literature}
\bibliographystyle{abbrv}

\onecolumn
\section{Appendix}
\subsection{Single Loop Gaussian Homotopy (SLGH)}
\cite{Iwakiri2022} proposes two single-loop Gaussian homotopy algorithms that applies $\nabla f(\bm{x})$, SLGHd and SLGHr. These two algorithms solves the surrogate problem of
\begin{equation} 
\label{slgh-goal}
\min_{\bm{\mu}\in\mathbb{R}^d,\sigma\in\mathbb{R}^+} F(\bm{\mu},\sigma),
\end{equation}
where
\begin{align*}
F(\bm{\mu},\sigma) &= \mathbb{E}_{\bm{w}\sim\mathcal{N}(\bm{\mu},\sigma^2 I_d)}\left[\mathbb{E}_{\bm{x}\sim \mathcal{D}}[f_{\bm{w}}(\bm{x})]\right] = \mathbb{E}_{\bm{x}\sim \mathcal{D}}[ \mathbb{E}_{\bm{w}\sim\mathcal{N}(\bm{\mu},\sigma^2 I_d)}[f_{\bm{w}}(\bm{x})]].
\end{align*}

We let
\begin{align*}
\bar{F}(\bm{\mu},\sigma;\bm{x})&= \mathbb{E}_{\bm{w}\sim\mathcal{N}(\bm{\mu},\sigma^2 I_d)}[f_{\bm{w}}(\bm{x})]\\
&= \mathbb{E}_{\bm{\epsilon}\sim \mathcal{N}(\bm{0},I_d)} [f_{\bm{\mu}+\sigma \bm{\epsilon}}(\bm{x})]\\
&= \frac{1}{\sqrt{2d}}\int_{\bm{\epsilon}\sim \mathcal{N}(\bm{0},I_d)} f_{\bm{\mu}+\sigma\bm{\epsilon}}(\bm{x}) e^{-\| \bm{\epsilon} \|^2} d\bm{\epsilon}.
\end{align*}
Under the assumption of interchangeable differentiation and expectation, we have
\begin{equation} 
\label{slgh-derivative}
\begin{split}
\nabla_{\bm{\mu}}\bar{F}(\bm{\mu},\sigma;\bm{x})&=\frac{1}{\sqrt{2d}}\int_{\bm{\epsilon}\sim \mathcal{N}(\bm{0},I_d)} \nabla f_{\bm{\mu}+\sigma\bm{\epsilon}}(\bm{x}) e^{-\| \bm{\epsilon} \|^2} d\bm{\epsilon} \\
& = \mathbb{E}_{\bm{\epsilon}\sim \mathcal{N}(\bm{0},I_d)}[\nabla f_{\bm{\mu}+\sigma\bm{\epsilon}}(\bm{x})] \in\mathbb{R}^d.
\end{split}
\end{equation}
\begin{equation}
\begin{split}
\nabla_{\sigma}\bar{F}(\bm{\mu},\sigma;\bm{x})&= \frac{1}{\sqrt{2d}}\int_{\bm{\epsilon}\sim \mathcal{N}(\bm{0},I_d)} \nabla f_{\bm{\mu}+\sigma\bm{\epsilon}}(\bm{x})\cdot \bm{\epsilon} e^{-\| \bm{\epsilon} \|^2} d\bm{\epsilon}\\
&= \mathbb{E}_{\bm{\epsilon}\sim \mathcal{N}(\bm{0},I_d)}[\nabla f_{\bm{\mu}+\sigma\bm{\epsilon}}(\bm{x})\cdot\bm{\epsilon}]\in\mathbb{R}, 
\end{split}
\end{equation}
where $\nabla f_{\bm{\mu}+\sigma\bm{\epsilon}}(\bm{x}):= \left.\frac{\partial f_{\bm{w}}(\bm{x})}{\partial\bm{w}}\right|_{\bm{w}=\bm{\mu}+\sigma\bm{\epsilon}}\in\mathbb{R}^d$, and $\cdot$ denotes the dot product between vectors. These two results imply that
\begin{align*} 
\nabla_{\bm{\mu}} F(\bm{\mu},\sigma) = \mathbb{E}_{\bm{x}\sim \mathcal{D}}[\nabla_{\bm{\mu}}\bar{F}(\bm{\mu},\sigma;\bm{x})] = \mathbb{E}_{\bm{x}\sim \mathcal{D},\bm{\epsilon}\sim \mathcal{N}(\bm{0},I_d)}[\nabla f_{\bm{\mu}+\sigma\bm{\epsilon}}(\bm{x})] \in\mathbb{R}^d.\\
\nabla_{\sigma} F(\bm{\mu},\sigma) = \mathbb{E}_{\bm{x}\sim \mathcal{D}}[\nabla_{\sigma}\bar{F}(\bm{\mu},\sigma;\bm{x})] = \mathbb{E}_{\bm{x}\sim\mathcal{D}, \bm{\epsilon}\sim \mathcal{N}(\bm{0},I_d)}[\nabla f_{\bm{\mu}+\sigma\bm{\epsilon}}(\bm{x})\cdot\bm{\epsilon}]\in\mathbb{R}. 
\end{align*}

\cite{Iwakiri2022} only states that $\nabla_{\bm{\mu}} F(\bm{\mu},\sigma)$ and $\nabla_{\sigma} F(\bm{\mu},\sigma)$ can be estimated by $\nabla_{\bm{\mu}} \bar{F}(\bm{\mu},\sigma;\bm{x})$ and $\nabla_{\sigma} \bar{F}(\bm{\mu},\sigma;\bm{x})$, using a fixed sample of $\bm{x}$. But they did not specify how the two estimators are computed or estimated given $\bm{x}$. Hence, it seems that they assume the two estimators can be computed exactly. Below, we provide a natural formula for approximating the two estimators. 

\begin{align*}
& \hat{\nabla}_{\bm{\mu}}\bar{F}(\bm{\mu},\sigma;\bm{x}) = \frac{1}{K}\sum_{k=1}^K \nabla f_{\bm{\mu}+\sigma \bm{\epsilon}_k}(\bm{x})\\
& \hat{\nabla}_{\sigma} \bar{F}(\bm{\mu},\sigma;\bm{x}) = \frac{1}{K}\sum_{k=1}^K \nabla f_{\bm{\mu}+\sigma\bm{\epsilon}_k}(\bm{x})\cdot \bm{\epsilon}_k\\
&\hat{\nabla}_{\bm{\mu}} F(\bm{\mu},\sigma) = \frac{1}{KJ}\sum_{j=1}^J \sum_{k=1}^K \nabla f_{\bm{\mu}+\sigma \bm{\epsilon}_k}(\bm{x}_j)\\
&\hat{\nabla}_{\bm{\mu}} F(\bm{\mu},\sigma) = \frac{1}{KJ}\sum_{j=1}^J \sum_{k=1}^K \nabla f_{\bm{\mu}+\sigma \bm{\epsilon}_k}(\bm{x}_j)\cdot\bm{\epsilon}_k
\end{align*}

\subsection{SLGHd}
At every time step $t$, the updating equation is
\begin{align}
\label{SLGH-d} 
\bm{\mu}_{t+1} &= \bm{\mu}_{t} + \beta \hat{\nabla}_{\bm{\mu}} F(\bm{\mu}_t,\sigma_t);\\
\sigma_{t+1} &= \max\{\min\{\sigma_t-\eta  \hat{\nabla}_{\sigma} F(\bm{\mu}_t,\sigma_t),\gamma \sigma_t\}, \epsilon'\}.
\end{align} 

\subsection{SLGHr}
At every time step $t$, the updating equation is
\begin{align}
\label{SLGH-d} 
\bm{\mu}_{t+1} &= \bm{\mu}_{t} + \beta \hat{\nabla}_{\bm{\mu}} F(\bm{\mu}_t,\sigma_t);\\
\sigma_{t+1} &= \gamma \sigma_t.
\end{align} 

\subsection{Proof to Theoretical Results}
\subsubsection{Proof to Lemma \ref{well-defined}}
\noindent\textbf{Lemma \ref{well-defined}} Under Assumption \ref{f-bound}, given any $N>0$ and $\sigma>0$, both $F_{N,\sigma}(\bm{\mu})$ and $\nabla_{\bm{\mu}}F_{N,\sigma}(\bm{\mu})$ are well-defined for all $\bm{\mu}\in\mathbb{R}^d$.
\begin{proof}

\begin{equation*}
\begin{split}
F_{N,\sigma}(\bm{\mu})=^{(\ref{surrogate})} &\mathbb{E}_{\bm{x}\sim \mathcal{D},\bm{w}\sim\mathcal{N}(\bm{\mu},\sigma^2I_d)}[f^{(N)}_{\bm{w}}(\bm{x})]\\
= &\mathbb{E}_{\bm{x}\sim \mathcal{D},\bm{w}\sim\mathcal{N}(\bm{\mu},\sigma^2I_d)}[e^{Nf_{\bm{w}}(\bm{x})}]\\
\leq &\mathbb{E}_{\bm{x}\sim \mathcal{D},\bm{w}\sim\mathcal{N}(\bm{\mu},\sigma^2I_d)}[e^{NM_f}]\\
= & e^{NM_f} <\infty.
\end{split}
\end{equation*}

\begin{equation*}
\begin{split}
\|\nabla F_{N,\sigma}(\bm{\mu}) \|&=^{\text{by (\ref{F-gradient})}} \left\|\mathbb{E}_{\bm{w}\sim \mathcal{N}(\bm{\mu}, \sigma^2I_d),\bm{x}\sim \mathcal{D}}\left[Ne^{Nf_{\bm{w}}(\bm{x})}\frac{\partial f_{\bm{w}}(\bm{x})}{\partial \bm{w}}\right] \right\|\\
&\leq Ne^{NM_f} \mathbb{E}_{\bm{w}\sim \mathcal{N}(\bm{\mu}, \sigma^2I_d),\bm{x}\sim \mathcal{D}}\left[\left\|\frac{\partial f_{\bm{w}}(\bm{x})}{\partial \bm{w}}\right\|\right] \\
&\leq Ne^{NM_f} L_{0f} <\infty,
\end{split}
\end{equation*}
where the last line is implied by the Lipschitz property of $f$ assumed in Assumption \ref{f-bound}. This finishes the proof for Lemma \ref{well-defined}.
\end{proof}

\subsubsection{Proof to Lemma \ref{F-smooth}}
\textbf{Lemma \ref{F-smooth}}
Under Assumption \ref{f-bound} and \ref{Lipschitz-smooth}, for any $\bm{\mu}_1$, $\bm{\mu}_2\in\mathbb{R}^d$, $\| \nabla F_{N,\sigma}(\bm{\mu}_1) -  \nabla F_{N,\sigma}(\bm{\mu}_2) \|\leq L_F\|\bm{\mu}_1-\bm{\mu}_2\|$, where $L_F=e^{NM_f}N(NL_{0f}^2+L_{1f})$.
\begin{proof}
Recall that
$$ \nabla F_{N,\sigma}(\bm{\mu}) =  \mathbb{E}_{\bm{w}\sim \mathcal{N}(\bm{\mu}, \sigma^2I_d),\bm{x}\sim \mathcal{D}}\left[Ne^{Nf_{\bm{w}}(\bm{x})}\frac{\partial f_{\bm{w}}(\bm{x})}{\partial \bm{w}}\right] = \mathbb{E}_{\bm{\epsilon}\sim \mathcal{N}(\bm{0}, I_d),\bm{x}\sim \mathcal{D}}\left[Ne^{Nf_{\bm{\mu}+\sigma \bm{\epsilon}}(\bm{x})}\left.\frac{\partial f_{\bm{w}}(\bm{x})}{\partial \bm{w}}\right|_{\bm{w}=\bm{\mu}+\sigma\bm{\epsilon}} \right]. $$
Then,
\begin{equation}
\begin{split}
\label{nabla-diff}
\nabla F_{N,\sigma}(\bm{\mu}_1) - \nabla F_{N,\sigma}(\bm{\mu}_2)  =& \mathbb{E}_{\bm{\epsilon}\sim \mathcal{N}(\bm{0}, I_d),\bm{x}\sim \mathcal{D}}\left[Ne^{Nf_{\bm{\mu}_1+\sigma\bm{\epsilon}}(\bm{x})} \nabla_{\bm{w}} f_{\bm{\mu}_1+\sigma\bm{\epsilon}}(\bm{x}) - Ne^{Nf_{\bm{\mu}_2+\sigma\bm{\epsilon}}(\bm{x})} \nabla_{\bm{w}} f_{\bm{\mu}_2+\sigma\bm{\epsilon}}(\bm{x})  \right] \\
= & \mathbb{E}_{\bm{\epsilon}\sim \mathcal{N}(\bm{0}, I_d),\bm{x}\sim \mathcal{D}}[(Ne^{Nf_{\bm{\mu}_1+\sigma\bm{\epsilon}}(\bm{x})} -Ne^{Nf_{\bm{\mu}_2+\sigma\bm{\epsilon}}(\bm{x})})\nabla_{\bm{w}} f_{\bm{\mu}_1+\sigma\bm{\epsilon}}(\bm{x})  \\
&\qquad\qquad\qquad\quad+ Ne^{Nf_{\bm{\mu}_2+\sigma\bm{\epsilon}}(\bm{x})} (\nabla_{\bm{w}} f_{\bm{\mu}_1+\sigma\bm{\epsilon}}(\bm{x})-\nabla_{\bm{w}} f_{\bm{\mu}_2+\sigma\bm{\epsilon}}(\bm{x}))] \\
\end{split}
\end{equation}
We bound the two terms inside $\mathbb{E}$ on the RHS separately. For the first term,
\begin{equation*}
\begin{split}
\| Ne^{Nf_{\bm{\mu}_1+\sigma\bm{\epsilon}}(\bm{x})} -Ne^{Nf_{\bm{\mu}_2+\sigma\bm{\epsilon}}(\bm{x})})\nabla_{\bm{w}} f_{\bm{\mu}_1+\sigma\bm{\epsilon}}(\bm{x}) \| & \leq N e^{\xi} N |f_{\bm{\mu}_1+\sigma\bm{\epsilon}}(\bm{x})-f_{\bm{\mu}_2+\sigma\bm{\epsilon}}(\bm{x})| L_{0f},\quad \text{by MVT},\\
&\leq N^2 e^{NM_f}L^2_{0f} \|\bm{\mu}_1 - \bm{\mu}_2\|,\quad \text{by Assumption \ref{f-bound}},
\end{split}
\end{equation*}
where $\xi$ denotes a number that lies between $Nf_{\bm{\mu}_1+\sigma\bm{\epsilon}}(\bm{x})$ and $Nf_{\bm{\mu}_2+\sigma\bm{\epsilon}}(\bm{x})$. For the second term,
$$\|Ne^{Nf_{\bm{\mu}_2+\sigma\bm{\epsilon}}(\bm{x})} (\nabla_{\bm{w}} f_{\bm{\mu}_1+\sigma\bm{\epsilon}}(\bm{x})-\nabla_{\bm{w}} f_{\bm{\mu}_2+\sigma\bm{\epsilon}}(\bm{x}))] \|\leq Ne^{NM_f}L_{1f}\|\bm{\mu}_1 - \bm{\mu}_2 \|,\quad\text{by Assumption \ref{f-bound} \& \ref{Lipschitz-smooth}}. $$
Plugging the two results to (\ref{nabla-diff}) gives
$$ \|\nabla F_{N,\sigma}(\bm{\mu}_1) - \nabla F_{N,\sigma}(\bm{\mu}_2)\|\leq e^{NM_f}N(NL_{0f}^2+L_{1f})\|\bm{\mu}_1-\bm{\mu}_2\|. $$
\end{proof}

\subsubsection{Proof to Lemma \ref{nablaF-bound}}
\textbf{Lemma \ref{nablaF-bound}} Under Assumption \ref{f-bound}, $\| \hat{\nabla} F_{N,\sigma}(\bm{\mu}) \|^2 \leq N^2e^{2NM_f} L_{0f}^2$.
\begin{proof}
\begin{equation}
\begin{split}
\| \hat{\nabla} F_{N,\sigma}(\bm{\mu}) \|^2 = & \frac{1}{K^2J^2} \sum_{k=1}^K\sum_{j=1}^J\sum_{l=1}^K\sum_{m=1}^J N^2e^{N(f_{\bm{w}_k}(\bm{x}_j)+f_{\bm{w}_l}(\bm{x}_m))} (\nabla_{\bm{w}} f_{\bm{w}_k}(\bm{x}_j))'\nabla_{\bm{w}}f_{\bm{w}_l}(\bm{x}_m)\\
\leq & \frac{N^2e^{2NM_f}}{K^2J^2} \sum_{k=1}^K\sum_{j=1}^J\sum_{l=1}^K\sum_{m=1}^J \|\nabla_{\bm{w}} f_{\bm{w}_k}(\bm{x}_j)\|\cdot\|\nabla_{\bm{w}}f_{\bm{w}_l}(\bm{x}_m)\|\\
\leq & N^2e^{2NM_f} L_{0f}^2,
\end{split}
\end{equation}
where the superscript $'$ denotes vector transpose.
\end{proof}

\subsubsection{Proof to Theorem \ref{main-theorem}}
\textbf{Theorem }\ref{main-theorem}. Let $\{\bm{\mu}_t\}_{t=1}^T$ be the series produced by (\ref{update-rule}) with a pre-selected deterministic $\bm{\mu}_0$. Under Assumption \ref{f-bound} and \ref{Lipschitz-smooth}, we have 
$$\sum_{t=1}^{T}\alpha_t \mathbb{E}[\|\nabla F_{N_{t},\sigma_{t}}(\bm{\mu}_{t-1})\|^2] \leq e^{(N_0+\Delta)M_f} -F_{N_{0},b+\sigma_{0}}(\bm{\mu}_{0}) + H_1 \sum_{t=0}^{T-1}\alpha_t^2 +H_2d \sum_{t=0}^{T-1}\beta^t+H_3 \sum_{t=1}^{T} \phi_{t},$$
where $H_1=(N_0+\Delta)^3 e^{3(N_0+\Delta)M_f}((N_0+\Delta)L^2_{0f}+L_{1f})L_{0f}^2$, $H_2 = e^{(N_0+\Delta)M_f}(N_0+\Delta)L_{0f}\sigma_0(1-\beta)\sqrt{2}$, $H_3:=\max\{ (N_0e)^{-1},  e^{(N_0+\Delta)M_f} M_f \}\Delta$.
\begin{proof}
\begin{equation}
\begin{split}
\label{EF-increment}
\mathbb{E}[F_{N_{t+1},\sigma_{t+1}}(\bm{\mu}_{t+1})]-\mathbb{E}[F_{N_{t},\sigma_{t}}(\bm{\mu}_{t})] = &\mathbb{E}[F_{N_{t+1},\sigma_{t+1}}(\bm{\mu}_{t+1})-F_{N_{t+1},\sigma_{t+1}}(\bm{\mu}_{t})] \\
&+ \mathbb{E}[F_{N_{t+1},\sigma_{t+1}}(\bm{\mu}_{t})-F_{N_{t+1},\sigma_{t}}(\bm{\mu}_{t})] + \mathbb{E}[F_{N_{t+1},\sigma_{t}}(\bm{\mu}_{t})-F_{N_{t},\sigma_{t}}(\bm{\mu}_{t})].
\end{split}
\end{equation}
We provide lower bounds for each of the three terms on the right-hand side (RHS) separately.\\

\textbf{Bound the first term.}
\begin{equation*}
\begin{split}
F_{N_{t+1},\sigma_{t+1}}(\bm{\mu}_{t+1})=&F_{N_{t+1},\sigma_{t+1}}(\bm{\mu}_t) + (\nabla F_{N_{t+1},\sigma_{t+1}}(\bm{\nu}_t))' (\bm{\mu}_{t+1}-\bm{\mu}_t), \quad \text{by MVT},\\
=&F_{N_{t+1},\sigma_{t+1}}(\bm{\mu}_t) + (\nabla F_{N_{t+1},\sigma_{t+1}}(\bm{\mu}_t))' (\bm{\mu}_{t+1}-\bm{\mu}_t) \\
&+ (\nabla F_{N_{t+1},\sigma_{t+1}}(\bm{\nu}_{t})-\nabla F_{N_{t+1},\sigma_{t+1}}(\bm{\mu}_t))'(\bm{\mu}_{t+1}-\bm{\mu}_t)\\
=&F_{N_{t+1},\sigma_{t+1}}(\bm{\mu}_t) + \alpha_{t+1} (\nabla F_{N_{t+1},\sigma_{t+1}}(\bm{\mu}_t))' \hat{\nabla} F_{N_{t+1},\sigma_{t+1}}(\bm{\mu}_t),\quad \text{by (\ref{update-rule})},\\
&+ (\nabla F_{N_{t+1},\sigma_{t+1}}(\bm{\nu}_t)-\nabla F_{N_{t+1},\sigma_{t+1}}(\bm{\mu}_t))'(\bm{\mu}_{t+1}-\bm{\mu}_t)\\
\geq &F_{N_{t+1},\sigma_{t+1}}(\bm{\mu}_t) + \alpha_{t+1} (\nabla F_{N_{t+1},\sigma_{t+1}}(\bm{\mu}_t))' \hat{\nabla} F_{N_{t+1},\sigma_{t+1}}(\bm{\mu}_t)\\
&-L_F\|\bm{v}_t-\bm{\mu}_t\|\cdot\|\bm{\mu}_{t+1}-\bm{\mu}_t \|,\quad \text{by Lemma \ref{F-smooth} and Cauchy Schwarz Ineq.},\\
\geq &F_{N_{t+1},\sigma_{t+1}}(\bm{\mu}_t) + \alpha_{t+1} (\nabla F_{N_{t+1},\sigma_{t+1}}(\bm{\mu}_t))' \hat{\nabla} F_{N_{t+1},\sigma_{t+1}}(\bm{\mu}_t) -L_F\|\bm{\mu}_{t+1}-\bm{\mu}_t \|^2,\\
= &F_{N_{t+1},\sigma_{t+1}}(\bm{\mu}_t) + \alpha_{t+1} (\nabla F_{N_{t+1},\sigma_{t+1}}(\bm{\mu}_t))' \hat{\nabla} F_{N_{t+1},\sigma_{t+1}}(\bm{\mu}_t)-\alpha_{t+1}^2L_F\|\hat{\nabla} F_{N_{t+1},\sigma_{t+1}}(\bm{\mu}_t)\|^2,
\end{split}
\end{equation*}
where $\bm{\nu}_t=\lambda_t\bm{\mu}_{t}+(1-\lambda_t)\bm{\mu}_{t+1}$ for some $\lambda_t\in [0,1]$ and $L_F = e^{N_{t+1}M_f}N_{t+1}(N_{t+1}L_{0f}^2+L_{1f})$. Taking the expectation on both the left-end and right-end gives:
\begin{equation*}
\begin{split}
\mathbb{E}[F_{N_{t+1},\sigma_{t+1}}(\bm{\mu}_{t+1})] &= \mathbb{E}[F_{N_{t+1},\sigma_{t+1}}(\bm{\mu}_{t})] +\alpha_{t+1} \mathbb{E}[\|\nabla F_{N_{t+1},\sigma_{t+1}}(\bm{\mu}_t)\|^2] -\alpha_{t+1}^2 L_F \mathbb{E}[\|\hat{\nabla} F_{N_{t+1},\sigma_{t+1}}(\bm{\mu}_t)\|^2]\\
&\geq  \mathbb{E}[F_{N_{t+1},\sigma_{t+1}}(\bm{\mu}_{t})] +\alpha_{t+1} \mathbb{E}[\|\nabla F_{N_{t+1},\sigma_{t+1}}(\bm{\mu}_t)\|^2] -\alpha_{t+1}^2 H_1, \quad\text{from Lemma \ref{nablaF-bound}},
\end{split}
\end{equation*}
where $H_1=(N_0+\Delta)^3 e^{3(N_0+\Delta)M_f}((N_0+\Delta)L^2_{0f}+L_{1f})L_{0f}^2$.
It further implies
\begin{equation}
\label{first-term}
\mathbb{E}[F_{N_{t+1},\sigma_{t+1}}(\bm{\mu}_{t+1})-F_{N_{t+1},\sigma_{t+1}}(\bm{\mu}_{t})] \geq \alpha_{t+1} \mathbb{E}[\|\nabla F_{N_{t+1},\sigma_{t+1}}(\bm{\mu}_t)\|^2] -\alpha_{t+1}^2 H_1.
\end{equation}

\textbf{Bound the second term.}
\begin{equation*}
\begin{split}
|F_{N_{t+1},\sigma_{t+1}}(\bm{\mu}_t) - F_{N_{t+1},\sigma_{t}}(\bm{\mu}_t)| 
&=\left|\mathbb{E}_{\bm{\epsilon}\sim\mathcal{N}(\bm{0},I_d),\bm{x}\sim\mathcal{D}} \left[e^{N_{t+1}f_{\bm{\mu}_t+\sigma_{t+1}\bm{\epsilon}}(\bm{x})}-e^{N_{t+1}f_{\bm{\mu}_t+\sigma_{t}\bm{\epsilon}}(\bm{x})}\right]\right| \\
&\leq \mathbb{E}_{\bm{\epsilon}\sim\mathcal{N}(\bm{0},I_d)}  \left| e^{\xi}N_{t+1} (f_{\bm{\mu}_t+\sigma_{t+1}\bm{\epsilon}}(\bm{x})-f_{\bm{\mu}_t+\sigma_{t}\bm{\epsilon}}(\bm{x}) ) \right|\\
&\leq  e^{(N_0+\Delta)M_f}(N_0+\Delta)\mathbb{E}_{\bm{\epsilon}\sim\mathcal{N}(\bm{0},I_d)}  [L_{0f}\|(\sigma_{t+1}-\sigma_t)\bm{\epsilon}\|]\\
&=  e^{(N_0+\Delta)M_f}(N_0+\Delta)L_{0f}\sigma_0(1-\beta)\beta^t\mathbb{E}_{\bm{\epsilon}\sim\mathcal{N}(\bm{0},I_d)}  \|\bm{\epsilon}\|\\
 &\leq e^{(N_0+\Delta)M_f}(N_0+\Delta)L_{0f}\sigma_0(1-\beta)\beta^t \frac{\sqrt{2}\Gamma ((d+1)/2) }{\Gamma(d/2)} \\
&\leq e^{(N_0+\Delta)M_f}(N_0+\Delta)L_{0f}\sigma_0(1-\beta)\beta^t \sqrt{2}d \\
&= H_2d \beta^t, 
\end{split}
\end{equation*}
where $H_2 = e^{(N_0+\Delta)M_f}(N_0+\Delta)L_{0f}\sigma_0(1-\beta)\sqrt{2}$ and $\Gamma$ denotes the Gamma function. This result implies
\begin{equation}
\label{second-term}
\mathbb{E}[F_{N_{t+1},\sigma_{t+1}}(\bm{\mu}_t) - F_{N_{t+1},\sigma_{t}}(\bm{\mu}_t)] \geq  - H_2d \beta^t.
\end{equation}

\textbf{Bound the third term.}
\begin{equation}
\label{F-in-third}
| F_{N_{t+1},\sigma_{t}}(\bm{\mu}_{t})-F_{N_{t},\sigma_{t}}(\bm{\mu}_{t}) | = |\mathbb{E}_{\bm{\varepsilon}\sim\mathcal{N}(\bm{0},I_d),\bm{x}\sim \mathcal{D}}[e^{N_{t+1}f_{\bm{\mu}_t+\sigma_t \bm{\varepsilon}}(\bm{x})}-e^{N_{t}f_{\bm{\mu}_t+\sigma_t \bm{\varepsilon}}(\bm{x})}]| \leq 
H_3 \phi_{t+1},
\end{equation}
where $H_3:=\max\{ (N_0e)^{-1},  e^{(N_0+\Delta)M_f} M_f \}\Delta$. Below, we provide the proof. For convenience, denote $f_{\bm{\mu}_t+\sigma_t \bm{\varepsilon}}(\bm{x})$ by $x$. On one hand, if $x<0$, then 
\begin{equation}
\begin{split}
|e^{N_{t+1}x}-e^{N_tx}| &= e^{N_tx} - e^{N_{t+1}x} \\
&= e^{\xi_t x} (-x) \phi_{t+1}\Delta,\quad\text{by MVT},\\
&\leq  e^{N_0 x} (-x) \phi_{t+1}\Delta\\
&\leq  (N_0e)^{-1} \phi_{t+1}\Delta,
\end{split}
\end{equation}
where $\xi_t\in [N_t,N_{t+1}]$ and the last line is because of $\max_{x\in\mathbb{R}}h(x)=(N_0e)^{-1}$ (by the first-order condition), where $h(x):=e^{N_0 x} (-x)$.
 
On the other hand, if $x\geq 0$ (implying $M_f>0$), 
$$ |e^{N_{t+1}x}-e^{N_tx}| = e^{\xi_tx} x\phi_{t+1}\Delta\leq e^{(N_0+\Delta)M_f} M_f \phi_{t+1}\Delta,$$
where the second inequality is because $x\leq M_f$ by Assumption \ref{f-bound}. 
In sum, for any $x\leq M_f$, we have that 
$$ |e^{N_{t+1}x}-e^{N_tx}| \leq \max\{ (N_0e)^{-1},  e^{(N_0+\Delta)M_f} M_f \}\phi_{t+1}\Delta.$$
This implies (\ref{F-in-third}), which in turn indicates
\begin{equation}
\label{third-term}
\mathbb{E}[F_{N_{t+1},\sigma_{t}}(\bm{\mu}_{t})- F_{N_{t},\sigma_{t}}(\bm{\mu}_{t})] \geq - H_3 \phi_{t+1}.
\end{equation}

\textbf{Wrap Up.}

Plugging (\ref{first-term}), (\ref{second-term}), and (\ref{third-term}) to (\ref{EF-increment}) gives
\begin{equation}
\begin{split}
\mathbb{E}[F_{N_{t+1},\sigma_{t+1}}(\bm{\mu}_{t+1})]-\mathbb{E}[F_{N_{t},\sigma_{t}}(\bm{\mu}_{t})] \geq  &\alpha_{t+1} \mathbb{E}[\|\nabla F_{N_{t+1},\sigma_{t+1}}(\bm{\mu}_t)\|^2] -\alpha_{t+1}^2 H_1-H_2d\beta^t- H_3\phi_{t+1}.
\end{split}
\end{equation}
Take the sum of both sides over $t\in\{0,1,...,T-1\}$ gives
\begin{equation}
\begin{split}
\mathbb{E}[F_{N_{T},\sigma_{T}}(\bm{\mu}_{T})] - \mathbb{E}[F_{N_{0},b+\sigma_{0}}(\bm{\mu}_{0})] \geq &\sum_{t=1}^{T}\alpha_{t} \mathbb{E}[\|\nabla F_{N_{t},\sigma_{t}}(\bm{\mu}_{t-1})\|^2] -H_1 \sum_{t=1}^{T}\alpha_t^2 -H_2d \sum_{t=0}^{T-1}\beta^t-H_3 \sum_{t=1}^{T} \phi_{t}.
\end{split}
\end{equation}
Reorganizing the terms gives
\begin{equation*}
\begin{split}
 \sum_{t=1}^{T}\alpha_t \mathbb{E}[\|\nabla F_{N_{t},\sigma_{t}}(\bm{\mu}_{t-1})\|^2] &\leq \mathbb{E}[F_{N_{T},\sigma_{T}}(\bm{\mu}_{T})] - \mathbb{E}[F_{N_{0},b+\sigma_{0}}(\bm{\mu}_{0})] + H_1 \sum_{t=0}^{T-1}\alpha_t^2 +H_2d \sum_{t=0}^{T-1}\beta^t+H_3 \sum_{t=1}^{T} \phi_{t}\\
&\leq  e^{(N_0+\Delta)M_f} - \mathbb{E}[F_{N_{0},b+\sigma_{0}}(\bm{\mu}_{0})] + H_1 \sum_{t=0}^{T-1}\alpha_t^2 +H_2d \sum_{t=0}^{T-1}\beta^t+H_3 \sum_{t=1}^{T} \phi_{t}.\\
&\leq  e^{(N_0+\Delta)M_f} - F_{N_{0},b+\sigma_{0}}(\bm{\mu}_{0}) + H_1 \sum_{t=0}^{T-1}\alpha_t^2 +H_2d \sum_{t=0}^{T-1}\beta^t+H_3 \sum_{t=1}^{T} \phi_{t}.
\end{split}
\end{equation*}
\end{proof}

\subsection{Proof to Corollary \ref{conv-rate}}
\textbf{Corollary \ref{conv-rate}}. Let $\{\bm{\mu}_t\}_{t=1}^T$ be the series produced by (\ref{update-rule}) with a pre-selected deterministic $\bm{\mu}_0$ and a learning rate $\alpha_t=t^{-(1/2+\gamma)}$ where $\gamma\in (0,1/2)$ and $T\geq 2$. Under Assumption \ref{lr-assumption}, \ref{f-bound}, and \ref{Lipschitz-smooth}, for any $\varepsilon>0$, whenever $T> (\frac{C_{0,\Delta}(1-2\gamma)}{2-2^{1/2+\gamma}}d\varepsilon^{-1})^{2/(1-2\gamma)}=O_{\Delta}((d \varepsilon^{-1})^{2/(1-2\gamma)})$, we have that
$$ \min_{t\in\{\lfloor T/2\rfloor,...,T\}} \mathbb{E}[\|\nabla F_{N_{t},\sigma_{t}}(\bm{\mu}_{t-1})\|^2] \leq \varepsilon. $$ Here, $C_{0,\Delta}:=e^{(N_0+\Delta)M_f} - F_{N_{0},b+\sigma_{0}}(\bm{\mu}_{0}) + H_1 \sum_{t=1}^{\infty}t^{-(1/2+\gamma)} +H_2(1-\beta)^{-1}+H_3$, $C_{1,\Delta}= \max\{1, 2/C_{0,\Delta}\}$, $\lfloor T/2 \rfloor$ denotes the largest integer no greater than $T/2$, and $\{H_1,H_2,H_3\}$ are as those defined in Theorem \ref{main-theorem}.

\begin{proof}
For any positive integer $t$, define
\begin{equation} 
\label{nu-def}
\nu_t:= \min_{\tau\in\{ \lfloor t/2\rfloor,...,t\}}\mathbb{E}[\|\nabla F_{N_{\tau},\sigma_{\tau}}(\bm{\mu}_{\tau-1})\|^2].
\end{equation}
Then, 
\begin{equation*}
\begin{split}
\sum_{t=\lfloor T/2\rfloor}^{T} \alpha_t \nu_{T} &\leq^{by (\ref{nu-def})} \sum_{t=\lfloor T/2\rfloor}^{T} \alpha_t  \mathbb{E}[\|\nabla F_{N_t,\sigma_{t}}(\bm{\mu}_{t-1})\|^2] \\
&\leq\sum_{t=1}^{T} \alpha_t  \mathbb{E}[\|\nabla F_{N_t,\sigma_{t}}(\bm{\mu}_{t-1})\|^2] \\
& \leq^{\text{Theorem \ref{main-theorem}}} e^{(N_0+\Delta)M_f} - F_{N_{0},b+\sigma_{0}}(\bm{\mu}_{0}) + H_1 \sum_{t=0}^{T-1}\alpha_t^2 +H_2d \sum_{t=0}^{T-1}\beta^t+H_3 \sum_{t=1}^{T} \phi_{t}\\
& \leq^{\text{Theorem \ref{main-theorem}}} e^{(N_0+\Delta)M_f} - F_{N_{0},b+\sigma_{0}}(\bm{\mu}_{0}) + H_1 \sum_{t=1}^{\infty}t^{-(1/2+\gamma)} +H_2d (1-\beta)^{-1}+H_3\\
& = C_{0,\Delta} d,
\end{split}
\end{equation*}
where the fourth line is because $e^{(N_0+\Delta)M_f} - F_{N_{0},b+\sigma_{0}}(\bm{\mu}_{0})>0$ and $d\geq 1$. Here, 
$$C_{0,\Delta}:=e^{(N_0+\Delta)M_f} - F_{N_{0},b+\sigma_{0}}(\bm{\mu}_{0}) + H_1 \sum_{t=1}^{\infty}t^{-(1/2+\gamma)} +H_2(1-\beta)^{-1}+H_3 <\infty.$$
Hence, 
$$\sum_{t=\lfloor T/2\rfloor}^{T} \alpha_t \nu_{T} \leq  C_{0,\Delta} d. $$
It further implies
\begin{equation*}
\begin{split}
\nu_T&\leq \frac{C_{0,\Delta} d}{\sum_{t=\lfloor T/2\rfloor}^{T} \alpha_t}\\
&\leq\frac{C_{0,\Delta}d}{\sum_{t=\lfloor T/2\rfloor}^{T} t^{-(1/2+\gamma)}}\\
&<\frac{C_{0,\Delta}d}{\int_{\lfloor T/2\rfloor}^Tt^{-(1/2+\gamma)}dt}\\
&=\frac{C_{0,\Delta}d}{(T^{\frac{1}{2}-\gamma}- \lfloor T/2\rfloor^{\frac{1}{2}-\gamma})/(\frac{1}{2}-\gamma)}\\
&\leq \frac{C_{0,\Delta}d}{(T^{\frac{1}{2}-\gamma}- (T/2)^{\frac{1}{2}-\gamma})/(\frac{1}{2}-\gamma)}\\
&= \frac{C_{0,\Delta}d}{ T^{1/2-\gamma}} \frac{1/2-\gamma}{1-(1/2)^{1/2-\gamma}}\\
&= \frac{C_{0,\Delta}d}{ T^{1/2-\gamma}} \frac{1-2\gamma}{2-2^{1/2+\gamma}}
\end{split}
\end{equation*}
Hence, we have
\begin{equation}
\label{nu-bound}
\nu_T \leq \frac{C_{0,\Delta}(1-2\gamma)}{2-2^{1/2+\gamma}}\frac{d}{T^{\frac{1}{2}-\gamma}}.
\end{equation}
Define $C_{1,\Delta}:=\max\{1, 2/C_{0,\Delta}\}$. Given any $\varepsilon\in(0,1)$, whenever $T> (\frac{C_{0,\Delta}(1-2\gamma)}{2-2^{1/2+\gamma}}d\varepsilon^{-1})^{2/(1-2\gamma)}=O_{\Delta}((d \varepsilon^{-1})^{2/(1-2\gamma)})$, we have
\begin{align*} \nu_T &\leq^{\text{from }(\ref{nu-bound})}  \frac{C_{0,\Delta}(1-2\gamma)}{2-2^{1/2+\gamma}}\frac{d}{T^{\frac{1}{2}-\gamma}}\\
&<  \frac{C_{0,\Delta}(1-2\gamma)}{2-2^{1/2+\gamma}} \frac{d}{(\frac{C_{0,\Delta}(1-2\gamma)}{2-2^{1/2+\gamma}}d\varepsilon^{-1})^{\frac{2}{1-2\gamma}(\frac{1}{2}-\gamma)}} \\
&=\varepsilon.
\end{align*}
This finishes the proof for Corollary \ref{conv-rate}.
\end{proof}

\subsection{Proof for Corollary \ref{Delta-irrelevant}}
\textbf{Corollary \ref{Delta-irrelevant}} The dependence of the $O$-term on $\Delta$ can be removed if we further assume that $M_f\leq 0$.

\begin{proof}
The dependence of the $O$-term in Corollary \ref{conv-rate} comes from the upper bound of $\sum_{t=1}^{T}\alpha_t \mathbb{E}[\|\nabla F_{N_{t},\sigma_{t}}(\bm{\mu}_{t-1})\|^2]$ in Theorem \ref{main-theorem}. 
Under the additional assumption that $M_f\leq 0$, we derive a new upper bound that is independent from $\Delta$. More specifically, we show that $H_1$, $H_2$, and $H_3$ in the upper bound can be replaced with terms that are irrelevant to $\Delta$. With this new bound, the $\Delta$-dependence of the $O$-term in Corollary \ref{conv-rate} can be removed.

Recall that $H_1=(N_0+\Delta)^3 e^{3(N_0+\Delta)M_f}((N_0+\Delta)L^2_{0f}+L_{1f})L_{0f}^2$ and $H_2 = e^{(N_0+\Delta)M_f}(N_0+\Delta)L_{0f}\sigma_0(1-\beta)\sqrt{2}$, both of which can be viewed as functions of $\Delta$. It is straightforward to see that
$  \lim_{\Delta \to +\infty} H_1(\Delta) = \lim_{\Delta \to +\infty} H_2(\Delta) = 0  $ whenever $M_f < 0$. Since both $  H_1(\Delta)  $ and $  H_2(\Delta)  $ are continuous and non-negative for $  \Delta \ge 0  $ (or $  \Delta > -N_0  $, depending on the domain), there exist finite constants $h_1, h_2 \geq0$, independent of $  \Delta$, such that $0 \leq H_1(\Delta) \leq h_1$ and $0 \le H_2(\Delta) \leq h_2$ for all $\Delta>0$.

Next, we take care of $H_3$. We replace the part of ``Bound the Third Term" in the proof for Theorem \ref{main-theorem} with the following, assuming $M_f\leq 0$. First, we show that
\begin{equation}
\label{H3-rederive}
| F_{N_{t+1},\sigma_{t}}(\bm{\mu}_{t})-F_{N_{t},\sigma_{t}}(\bm{\mu}_{t}) | = |\mathbb{E}_{\bm{\varepsilon}\sim\mathcal{N}(\bm{0},I_d),\bm{x}\sim \mathcal{D}}[e^{N_{t+1}f_{\bm{\mu}_t+\sigma_t \bm{\varepsilon}}(\bm{x})}-e^{N_{t}f_{\bm{\mu}_t+\sigma_t \bm{\varepsilon}}(\bm{x})}]| \leq 
((\sum_{\tau=1}^t\phi_\tau)e)^{-1} \phi_{t+1},
\end{equation}
For convenience, denote $f_{\bm{\mu}_t+\sigma_t \bm{\varepsilon}}(\bm{x})$ by $x$. Since $x\leq0$, then for any $t\geq 1$,
\begin{align*}
|e^{N_{t+1}x}-e^{N_tx}| &= e^{N_tx} - e^{N_{t+1}x} \\
&= e^{\xi_t x} (-x) \phi_{t+1}\Delta,\quad\text{by MVT},\\
&\leq  e^{N_t x} (-x) \phi_{t+1}\Delta\\
&\leq  (N_t e)^{-1} \phi_{t+1}\Delta\\
&= ((N_0+\sum_{\tau=1}^t\phi_\tau \Delta)e)^{-1} \phi_{t+1}\Delta,\quad\text{from \ref{update-rule}},\\
&\leq  \frac{\phi_{t+1}}{\sum_{\tau=1}^t \phi_\tau e}\\
&\leq  \phi_1^{-1} \phi_{t+1}
\end{align*}
where $\xi_t\in [N_t,N_{t+1}]$ and the last line is because of $\max_{x\in\mathbb{R}}h(x)=(N_te)^{-1}$ (by the first-order condition), where $h(x):=e^{N_t x} (-x)$. This proves (\ref{H3-rederive}), which in turn indicates
\begin{equation}
\mathbb{E}[F_{N_{t+1},\sigma_{t}}(\bm{\mu}_{t})- F_{N_{t},\sigma_{t}}(\bm{\mu}_{t})] \geq - \phi_1^{-1} \phi_{t+1}.
\end{equation}
Combining this result with the rest of the proof to Theorem \ref{main-theorem} gives
\begin{align*}
\sum_{t=2}^{T}\alpha_t \mathbb{E}[\|\nabla F_{N_{t},\sigma_{t}}(\bm{\mu}_{t-1})\|^2] &\leq e^{(N_0+\Delta)M_f} - \mathbb{E}[F_{N_{1},b+\sigma_{1}}(\bm{\mu}_{1})] + H_1 \sum_{t=1}^{T-1}\alpha_t^2 +H_2d \sum_{t=1}^{T-1}\beta^t+\phi_1^{-1}\sum_{t=2}^{T}\phi_{t}\\
&\leq 1 + h_1\sum_{t=0}^{T-1}\alpha_t^2 +h_2 d \sum_{t=0}^{T-1}\beta^t+\phi_1^{-1}\sum_{t=2}^{T}\phi_{t}.
\end{align*}
This new upper bound is independent from $\Delta$. Using this result to modify the proof to Corollary \ref{conv-rate} will remove the dependence of the $O$-term on $\Delta$.
\end{proof}

\subsection{Proof to Lemma \ref{sn-lemma}}
\textbf{Lemma \ref{sn-lemma}} Under Assumption \ref{f-bound} and \ref{ess}, for any $\delta>0$, we have
(1) $\lim_{N\rightarrow +\infty}s_N/v_N=0$, where $s_N:=\sup_{\|\bm{w}-\bm{w}^*\|\geq\delta}G_N(\bm{w})$, $v_N:=G_N(\bm{w}^*)e^{-N\eta/4}$, and $\eta:= \Psi(\bm{w}^*)-\sup_{\|\bm{w}-\bm{w}^*\|\geq \delta} \Psi(\bm{w})$, and (2) there exists $\delta'\in(0,\delta)$ and $N_{\delta'}$ such that $G_N(\bm{w})>v_N$ whenever $\|\bm{w}-\bm{w}^*\|<\delta'$ and $N>\frac{8}{\eta}\log c_{\epsilon}^{-1}$. Here, $\epsilon=\eta/16$ and $c_{\epsilon}$ is specified in Assumption \ref{ess}.
\begin{proof}
We first prove the following preliminary result: for any $\bm{w}$,
\begin{equation}
\label{GN-limit}
\lim_{N\rightarrow +\infty} \frac{1}{N}\log G_N(\bm{w}) = \Psi(\bm{w}).
\end{equation}
One one hand, for any $\varepsilon>0$, we have 
\begin{equation*}
\begin{split}
& G_N(\bm{w}) = \mathbb{E}_{\bm{x}\sim \mathcal{D}}[e^{Nf_{\bm{w}}(\bm{x})}] \geq e^{N(\Psi(\bm{w})-\varepsilon)}\mathbb{P}(f_{\bm{w}}(\bm{x})\geq \Psi(\bm{w})-\varepsilon) \\
\Rightarrow\quad & \frac{1}{N}\log G_N(\bm{w}) \geq \Psi(\bm{w})-\varepsilon + \frac{1}{N}\log \mathbb{P}(f_{\bm{w}}(\bm{x})\geq \Psi(\bm{w})-\varepsilon)\\
\Rightarrow\quad &\liminf_{N\rightarrow \infty} \frac{1}{N}\log G_N(\bm{w}) \geq \Psi(\bm{w})-\varepsilon
\end{split}
\end{equation*}
where the $\log\mathbb{P}$ term is well-defined since $\mathbb{P}(f_{\bm{w}}(\bm{x})>\Psi(\bm{w})-\varepsilon)>0$\footnote{If it does not hold true, then there exists some $\varepsilon_0>0$ such that $\mathbb{P}(f_{\bm{w}}(\bm{x})>\Psi(\bm{w})-\varepsilon_0)=0$. Hence, $\mathbb{P}(f_{\bm{w}}(\bm{x})\leq \Psi(\bm{w})-\varepsilon_0)=1$, which further implies $\inf\{a\in\mathbb{R} | \mathbb{P}(f_{\bm{w}}(\bm{x})\leq a)=1\}\leq\Psi(\bm{w})-\varepsilon_0$. That is, $\Psi(\bm{w})\leq \Psi(\bm{w})-\varepsilon_0$, which is a contradiction.}. 
On the other hand, since $\mathbb{P}(f_{\bm{w}}(\bm{x})\leq \Psi(\bm{w}))=1$,
\begin{equation*}
\begin{split}
 & \frac{1}{N}\log G_N(\bm{w}) \leq \frac{1}{N}\log \mathbb{E}_{\bm{x}\sim \mathcal{D}}[e^{N\Psi({\bm{w}})}] = \Psi({\bm{w}}) \\
\Rightarrow\quad & \limsup_{N\rightarrow \infty}  \frac{1}{N}\log G_N(\bm{w})  \leq \Psi(\bm{w}).
\end{split}
\end{equation*}
In sum, for any $\varepsilon>0$, we have
$$ \Psi(\bm{w})-\varepsilon\leq \liminf_{N\rightarrow \infty} \frac{1}{N}\log G_N(\bm{w}) \leq  \limsup_{N\rightarrow \infty}  \frac{1}{N}\log G_N(\bm{w})  \leq \Psi(\bm{w}). $$
Since $\varepsilon$ can be arbitrarily close to zero, we have that
$$ \liminf_{N\rightarrow \infty} \frac{1}{N}\log G_N(\bm{w})   =\limsup_{N\rightarrow \infty}  \frac{1}{N}\log G_N(\bm{w}) = \Psi(\bm{w}).$$
This further implies (\ref{GN-limit}).

Next, we prove the following limit:
\begin{equation}
\label{sN-limit}
\lim_{N\rightarrow \infty}\frac{1}{N}\log s_N = \sup_{\|\bm{w}-\bm{w}^*\|\geq\delta} \Psi(\bm{x}).
\end{equation} 
We denote $\mathcal{C}:=\{\bm{w}\in\mathbb{R}^d | \|\bm{w}-\bm{w}^*\|\geq\delta\}$. Then, it is a closed set in $\mathbb{R}^d$. On one hand,
\begin{equation*}
\begin{split}
&G_N(\bm{w}) = \mathbb{E}_{\bm{x}\sim \mathcal{D}}[e^{Nf_{\bm{w}}(\bm{x})}] \leq e^{N\Psi(\bm{w})} \leq e^{N\sup_{\bm{w}\in\mathcal{C}}\Psi(\bm{w})},\quad \text{for any } \bm{w}\in\mathcal{C},\\
\Rightarrow \quad &\sup_{\bm{w}\in\mathcal{C}} G_N(\bm{w}) \leq e^{N\sup_{\bm{w}\in\mathcal{C}}\Psi(\bm{w})} \quad  \Rightarrow\quad s_N \leq e^{N\sup_{\bm{w}\in\mathcal{C}}\Psi(\bm{w})} \\
\Rightarrow \quad &\frac{1}{N}\log s_N \leq \sup_{\bm{w}\in\mathcal{C}}\Psi(\bm{w})\\
\Rightarrow\quad &\limsup_{N\rightarrow \infty} \frac{1}{N}\log s_N \leq \sup_{\bm{w}\in\mathcal{C}}\Psi(\bm{w}).\\
\end{split}
\end{equation*}
On the other hand, for any $\varepsilon>0$, let $\bm{w}_\epsilon\in\mathcal{C}$ be such that $\Psi(\bm{w}_c)>\sup_{\bm{w}\in\mathcal{C}}\Psi(\bm{w})-\varepsilon$. Then,
\begin{equation*}
\begin{split}
&\frac{1}{N}\log s_N \geq \frac{1}{N}\log G_N(\bm{w}_c) \\
\Rightarrow \quad & \liminf_{N\rightarrow \infty} \frac{1}{N}\log s_N \geq  \liminf_{N\rightarrow \infty} \frac{1}{N}\log G_N(\bm{w}_c) = ^{\text{by (\ref{GN-limit})}} \Psi(\bm{w}_c)> \sup_{\bm{w}\in\mathcal{C}}\Psi(\bm{w})-\varepsilon.
\end{split}
\end{equation*}
In sum, we have
$$ \Psi(\bm{w})-\varepsilon< \liminf_{N\rightarrow \infty} \frac{1}{N}\log s_N\leq \limsup_{N\rightarrow \infty} \frac{1}{N}\log s_N \leq \sup_{\bm{w}\in\mathcal{C}}\Psi(\bm{w}),$$
which implies 
$$ \liminf_{N\rightarrow \infty} \frac{1}{N}\log s_N = \limsup_{N\rightarrow \infty} \frac{1}{N}\log s_N = \sup_{\bm{w}\in\mathcal{C}}\Psi(\bm{w}) $$
and hence (\ref{sN-limit}).

Now, we prove (1) in the lemma.
$$\lim_{N\rightarrow\infty} \frac{1}{N} (\log S_N -\log G_N(\bm{w}^*)) =^{\text{by (\ref{GN-limit})}, (\ref{sN-limit})} \sup_{\bm{w}\in\mathcal{C}}\Psi(\bm{w}) - \Psi(\bm{w}^*) =-\eta<^{\text{by Assumption \ref{ess}}}0. $$
This implies that, when $N$ is sufficiently large, we have $\frac{1}{N} (\log S_N -\log G_N(\bm{w}^*))<-\eta/2$, which further implies
\begin{equation*}
\begin{split}
\frac{s_N}{v_N} = \frac{s_N}{G_N(\bm{w}^*)e^{-N\eta/4}} = \left( e^{ \frac{1}{N}(\log S_N -\log G_N(\bm{w}^*))} \right)^Ne^{N\eta/4} \leq e^{-N\eta/4}. 
\end{split}
\end{equation*}
In sum, we have $\frac{s_N}{v_N} \in (0,e^{-N\eta/4})$ for sufficiently large $N$. By squeeze theorem, $\lim_{N\rightarrow \infty} \frac{s_N}{v_N}=0$, which is (1) in the lemma.

Next, we prove (2). On one hand,
\begin{equation*}
\begin{split}
&v_N=G_N(\bm{w}^*)e^{-N\eta/4} = \mathbb{E}[e^{Nf_{\bm{w}^*}(\bm{x})}]e^{-N\eta/4} \leq e^{N\Psi(\bm{w}^*)}e^{-N\eta/4}.
\end{split}
\end{equation*}
On the other hand, let $\epsilon:=\eta/16$. By (3) in Assumption \ref{ess}, there exists $\delta_{\epsilon}>0$ and $c_{\epsilon}\in(0,1)$ such that $\inf_{\|\bm{w}-\bm{w}^*\|<\delta_{\epsilon}}\mathbb{P}(f_{\bm{w}}(\bm{x})>\Psi(\bm{w})-\epsilon)>c_{\epsilon}$. Also, by the continuity of $\Psi$ in Assumption \ref{ess}, there exists $\delta_1>0$ such that $\Psi(\bm{w})>\Psi(\bm{w}^*)-\epsilon$ whenever $\|\bm{w}-\bm{w}^*\|<\delta_1$. Define $\delta'=\min\{\delta,\delta_{\epsilon},\delta_1\}$. Then, whenever $\|\bm{w}-\bm{w}^*\|<\delta'$,
\begin{equation*}
\begin{split}
& G_N(\bm{w}) = \mathbb{E}[e^{Nf_{\bm{w}}(\bm{x})}] \geq e^{N(\Psi(\bm{w})-\epsilon)}c_{\epsilon} \geq e^{N(\Psi(\bm{w}^*)-2\epsilon)}c_{\epsilon}=e^{N\Psi(\bm{w}^*)}e^{-N\eta/4)}e^{N\eta/8}c_{\epsilon}.
\end{split}
\end{equation*}
Therefore, 
$$  G_N(\bm{w}) \geq e^{N\Psi(\bm{w}^*)}e^{-N\eta/4)}e^{N\eta/8}c_{\epsilon}\geq v_N e^{N\eta/8}c_{\epsilon}.$$
When $e^{N\eta/8}c_{\epsilon}>1$ (i.e., $N>\frac{8}{\eta}\log c_{\epsilon}^{-1}$), $G_N(\bm{w}) >v_N$. This finishes the proof for (2) in the lemma.
\end{proof}

\subsubsection{Proof to Theorem \ref{main-convergence-thm}}
\textbf{Theorem }\ref{main-convergence-thm}
Suppose Assumption \ref{f-bound} and \ref{ess} hold. Denote $\bm{\mu}=[\mu_1,...,\mu_d]$ and $\bm{w}^*=[w_1^*,...,w_d^*]$. For any $\delta>0$, $\sigma>0$, and $M>0$, there exists a threshold $N_{\delta,\sigma,M}>0$ such that, whenever $N>N_{\delta,\sigma,M}$ and $\|\bm{\mu}\|\leq M$, for all $i\in\{1,2,...,d\}$ we have $\frac{\partial F_{N,\sigma}(\bm{\mu})}{\partial \mu_i}<0$ if $\mu_i>w_i^*+\delta$, and $\frac{\partial F_{N,\sigma}(\bm{\mu})}{\partial \mu_i}>0$ if $\mu_i<w_i^*-\delta$. 
\begin{proof}
We re-write $F_{N,\sigma}(\bm{\mu})$ in (\ref{surrogate}) as
\begin{equation}
\label{F-decompose}
F_{N,\sigma}(\bm{w}) = v_N (H_{N,,\sigma}(\bm{\mu}) + R_{N,,\sigma}(\bm{\mu})),
\end{equation}
where
\begin{equation*}
\begin{split}
\label{theorem-notations}
&G_N(\bm{w})=\mathbb{E}_{\bm{x}\sim \mathcal{D}}[e^{Nf_{\bm{w}}(\bm{x})}],\\
&H_{N,\sigma}(\bm{\mu}):=  (\sqrt{2\pi}\sigma)^{-d}  \int_{\bm{w}\in B(\bm{w}^*;\delta)} v^{-1}_N G_N(\bm{w})  e^{-\frac{\lVert \bm{w} - \bm{\mu} \rVert^2}{2\sigma^2}} d\bm{w},\\
&R_{N,\sigma}(\bm{\mu}):=   (\sqrt{2\pi}\sigma)^{-d} \int_{\bm{w}\notin B(\bm{w}^*;\delta)} v^{-1}_N G_N(\bm{w})  e^{-\frac{\lVert \bm{w} - \bm{\mu} \rVert^2}{2\sigma^2}} d\bm{w},\\
\end{split}
\end{equation*}
and $B(\bm{w}^*;\delta):=\{\bm{w}\in\mathbb{R}^d: \|\bm{w}-\bm{w}^*\|<\delta \}$. The core idea of the proof is to first establish that, for sufficiently large $N$, $\left| \frac{\partial H_{N,\sigma}(\bm{\mu})}{\partial\mu_i}\right|$ dominates $\left| \frac{\partial R_{N,\sigma}(\bm{\mu})}{\partial\mu_i}\right|$, and then to show that $\frac{\partial H_N(\bm{\mu},\sigma)}{\partial\mu_i}$ has the sign claimed in the theorem.

First, we bound $\left| \frac{\partial R_{N,\sigma}(\bm{\mu})}{\partial\mu_i}\right|$ from above. For any $\bm{\mu}\in \mathbb{R}^d$,
\begin{equation}
\begin{split}
\label{R-bound}
\left| \frac{\partial R_{N,\sigma}(\bm{\mu})}{\partial\mu_i}\right|&=\left| \frac{1}{(\sqrt{2\pi})^d\sigma^{d+2}}\int_{\bm{w}\notin B(\bm{w}^*;\delta)} (w_i - \mu_i) e^{-\frac{\lVert \bm{w} - \bm{\mu} \rVert^2}{2\sigma^2}} v_N^{-1} G_N(\bm{w}) d\bm{w} \right| \\
&\leq \frac{1}{(\sqrt{2\pi})^d\sigma^{d+2}}\int_{\bm{w}\notin B(\bm{w}^*;\delta)} |w_i - \mu_i| e^{-\frac{\lVert \bm{w} - \bm{\mu} \rVert^2}{2\sigma^2}} v_N^{-1} G_N(\bm{w}) d\bm{w} \\
&\leq \frac{1}{(\sqrt{2\pi})^d\sigma^{d+2}}\int_{\bm{w}\notin B(\bm{w}^*;\delta)} |w_i - \mu_i| e^{-\frac{\lVert \bm{w} - \bm{\mu} \rVert^2}{2\sigma^2}} v_N^{-1} s_N d\bm{w} \\
&\leq s_Nv_N^{-1} \left( \Pi_{j\neq i}  \frac{1}{\sqrt{2\pi}\sigma} \int_{w_j\in\mathbb{R}} 
e^{-\frac{(w_j - \mu_j )^2}{2\sigma^2}}  dw_j \right)  \cdot\frac{1}{\sqrt{2\pi}\sigma^{3}}\int_{w_i\in \mathbb{R}}  |w_i - \mu_i| e^{-\frac{(w_i - \mu_i)^2}{2\sigma^2}}  dw_i\\
&=s_Nv_N^{-1} \frac{1}{\sqrt{2\pi}\sigma^{3}}\int_{y\in \mathbb{R}} \sqrt{2}\sigma|y| e^{-y^2}  d(\sqrt{2}\sigma y), \quad y:=\frac{w_i-\mu_i}{\sqrt{2}\sigma},\\
&=s_Nv_N^{-1} \frac{\sqrt{2}}{\sqrt{\pi}\sigma}\cdot 2\int_{0}^{\infty}y e^{-y^2}  dy, \\
&=s_Nv_N^{-1} \frac{\sqrt{2}}{\sqrt{\pi}\sigma}\cdot \int_{0}^{\infty}  e^{-y^2}  d y^2,\\
&=s_Nv_N^{-1} \frac{\sqrt{2}}{\sqrt{\pi}\sigma}\cdot \int_{0}^{\infty}  e^{-z}  d z,\\
&=\frac{\sqrt{2}}{\sqrt{\pi}\sigma}\cdot\frac{s_N}{v_N}.
\end{split}
\end{equation}

Second, we bound $\left|\frac{\partial H_N(\bm{\mu},\sigma)}{\partial\mu_i}\right|$ from below when $\|\bm{\mu}\|\leq M$ and $|\mu_i-w_i^*|>\delta$. 
\begin{equation}
\begin{split}
\label{H-bound}
\left| \frac{\partial H_{N,\sigma}(\bm{\mu})}{\partial\mu_i} \right|
&= \frac{1}{(\sqrt{2\pi})^d\sigma^{d+2}}\int_{\bm{w}\in B(\bm{w}^*;\delta)} |w_i - \mu_i| v_N^{-1} G_N(\bm{w})   e^{-\frac{\lVert \bm{w} - \bm{\mu} \rVert^2}{2\sigma^2}} d\bm{w} \\
&\geq \frac{1}{(\sqrt{2\pi})^d\sigma^{d+2}}\int_{\bm{w}\in B(\bm{w}^*;\delta')} |w_i - \mu_i|   e^{-\frac{\lVert \bm{w} - \bm{\mu} \rVert^2}{2\sigma^2}} d\bm{w}, \quad\text{by Lemma \ref{sn-lemma} (2)}, \\
&\geq  \frac{1}{(\sqrt{2\pi})^d\sigma^{d+2}}\int_{\bm{w}\in B(\bm{w}^*;\delta')} (\delta-\delta') e^{-\frac{\lVert \bm{w} - \bm{\mu} \rVert^2}{2\sigma^2}}  d\bm{w}\\
&\geq   \frac{1}{(\sqrt{2\pi})^d\sigma^{d+2}}\int_{\bm{w}\in B(\bm{w}^*;\delta')} (\delta-\delta')e^{-\frac{M^2}{\sigma^2}}e^{-\frac{\|\bm{w}\|^2}{\sigma^2}} d\bm{w},\quad\|  \bm{w} - \bm{\mu}\|^2\leq 2(\| \bm{w} \|^2+ \|\bm{\mu}\|^2),\\
&\geq  (\delta-\delta') e^{-\frac{M^2}{\sigma^2}}  V(\delta',d,\sigma) 
\end{split}
\end{equation}
where $V(\delta',d,\sigma):=\frac{1}{(\sqrt{2\pi})^d\sigma^{d+2}} \int_{\bm{w}\in B(\bm{w}^*;\delta')} e^{-\frac{\|\bm{w}\|^2}{\sigma^2}} d\bm{w}$ and the third line is because of
$$ |w_i-\mu_i|=|(w_i-w_i^*)+(w_i^*-\mu_i)|\geq |w_i^*-\mu_i|-|w_i-w_i^*|>\delta-\delta'. $$

Third, we show the dominance of $\left| \frac{\partial H_{N,\sigma}(\bm{\mu})}{\partial\mu_i}\right|$ over $\left| \frac{\partial R_{N,\sigma}(\bm{\mu})}{\partial\mu_i}\right|$. From Lemma \ref{sn-lemma} (1), $s_N/v_N$ on the right-end of (\ref{R-bound}) approaches 0 as $N\rightarrow \infty$. Hence, for $\varepsilon:= (\delta-\delta') e^{-\frac{M^2}{\sigma^2}}  V(\delta',d,\sigma) \sqrt{\pi}\sigma/\sqrt{2}$, there exists a threshold $N_{\delta,d,M}>0$ such that, whenever $N>N_{\delta,d,M}$, we have $s_N/v_N < \varepsilon$, which further implies
\begin{equation} 
\label{H>R}
\frac{\sqrt{2}}{\sqrt{\pi}\sigma}\cdot\frac{s_N}{v_N}< (\delta-\delta') e^{-\frac{M^2}{\sigma^2}}  V(\delta',d,\sigma) \quad \Rightarrow^{\text{by (\ref{H-bound}),\ref{R-bound}}}\;\; \left| \frac{\partial R_{N,\sigma}(\bm{\mu})}{\partial\mu_i}\right|< \left| \frac{\partial H_{N,\sigma}(\bm{\mu})}{\partial\mu_i} \right|.
\end{equation}

Finally, we prove that $\frac{\partial F_{N,\sigma}(\bm{\mu})}{\partial\mu_i}$ has the sign claimed in the theorem. When $N>N_{\delta,\sigma,M}$, $\|\bm{\mu}\|\leq M$, and $\mu_i>w_i^*+\delta$,
\begin{equation}
\begin{split}
\label{g-bound1}
\frac{1}{v_N}\frac{\partial F_{N,\sigma}(\bm{\mu})}{\partial\mu_i} &= \frac{\partial H_{N,\sigma}(\bm{\mu})}{\partial\mu_i}+ \frac{\partial R_{N,\sigma}(\bm{\mu})}{\partial\mu_i} \\
&= -\left|\frac{\partial H_N(\bm{\mu},\sigma)}{\partial\mu_i}\right| + \frac{\partial R_{N,\sigma}(\bm{\mu})}{\partial\mu_i} \\
&<^{\text{by }(\ref{H>R})} -\left|\frac{\partial R_{N,\sigma}(\bm{\mu})}{\partial\mu_i}\right| + \left|\frac{\partial R_{N,\sigma}(\bm{\mu})}{\partial\mu_i} \right|\\
&= 0,
\end{split}
\end{equation}
where $\frac{\partial H_{N,\sigma}(\bm{\mu})}{\partial\mu_i}=-\left|\frac{\partial H_N(\bm{\mu},\sigma)}{\partial\mu_i}\right|$ because $\frac{\partial H_{N,\sigma}(\bm{\mu})}{\partial\mu_i}=\frac{1}{(\sqrt{2\pi})^d\sigma^{d+2}}\int_{\bm{w}\in B(\bm{w}^*;\delta)} (w_i - \mu_i) v_N^{-1}G_N(\bm{w}) e^{-\frac{\lVert \bm{w} - \bm{\mu} \rVert^2}{2\sigma^2}}    d\bm{w}$ and the integrand is always negative when $\bm{w}\in B(\bm{w}^*;\delta)$.

On the other hand, when $N>N_{\delta,\sigma,M}$, $\|\bm{\mu}\|\leq M$, and $\mu_i<w_i^*-\delta$,
\begin{equation}
\begin{split}
\label{g-bound2}
\frac{1}{v_N}\frac{\partial F_{N,\sigma}(\bm{\mu})}{\partial\mu_i} &= \frac{\partial H_{N,\sigma}(\bm{\mu})}{\partial\mu_i}+ \frac{\partial R_{N,\sigma}(\bm{\mu})}{\partial\mu_i} \\
&= \left|\frac{\partial H_N(\bm{\mu},\sigma)}{\partial\mu_i}\right| + \frac{\partial R_{N,\sigma}(\bm{\mu})}{\partial\mu_i} \\
&>^{\text{by }(\ref{H>R})} \left|\frac{\partial R_{N,\sigma}(\bm{\mu})}{\partial\mu_i}\right| - \left|\frac{\partial R_{N,\sigma}(\bm{\mu})}{\partial\mu_i} \right|\\
&= 0.
\end{split}
\end{equation}
Then, (\ref{g-bound1}) and (\ref{g-bound2}) imply the result in the theorem since $v_N>0$ for all positive $N$. 
\end{proof}

\end{document}